\pdfoutput=1
\RequirePackage[l2tabu, orthodox]{nag}
\documentclass{amsart}
\usepackage[T1]{fontenc}
\usepackage[utf8]{inputenc}

      \title{Assembly maps for topological cyclic homology of group algebras}

     \author{Wolfgang Lück}
    \address{HIM (Hausdorff Research Institute for Mathematics)
             \&\ Mathematisches Institut, Rheinische Friedrich-Wilhelms-Universität Bonn, Germany}
      \email{\hemail{wolfgang.lueck@him.uni-bonn.de}}
    \urladdr{\hurl{him.uni-bonn.de/lueck/}}

     \author{Holger Reich}
    \address{Institut für Mathematik, Freie Universität Berlin, Germany}
      \email{\hemail{holger.reich@fu-berlin.de}}
    \urladdr{\hurl{mi.fu-berlin.de/math/groups/top/members/Professoren/reich.html}}

     \author{John Rognes}
    \address{Department of Mathematics, University of Oslo, Norway}
      \email{\hemail{rognes@math.uio.no}}
    \urladdr{\hurl[]{folk.uio.no/rognes/home.html}}

     \author{Marco Varisco}
    \address{Department of Mathematics and Statistics, University at Albany, SUNY, USA}
      \email{\hemail{mvarisco@albany.edu}}
    \urladdr{\hurl{albany.edu/~mv312143/}}

  \subjclass[2010]{\MSC{19D55}, \MSC{55P91}, \MSC{55P42}, \MSC{18G60}}

       \date{April 19, 2017}

\usepackage{amssymb}
\usepackage{aliascnt}
\usepackage{mathtools}  \mathtoolsset{showonlyrefs}
\usepackage{enumitem}
\usepackage{etoolbox}
\usepackage{mfirstuc}
\usepackage{pifont}
\usepackage{tikz}       \usetikzlibrary{cd}
\usepackage[shortcuts]{extdash}
\usepackage[pagebackref, pdfinfo={
      Title={Assembly maps for topological cyclic homology of group algebras},
     Author={Lück, Reich, Rognes, Varisco},
    Subject={MSC2010: 19D55, 55P91, 55P42, 18G60}
}]{hyperref}
\usepackage[alphabetic, backrefs, msc-links, nobysame]{amsrefs}

\makeatletter
\def\@seccntformat#1{%
  \protect\textup{%
    \protect\@secnumfont
    \expandafter\protect\csname format#1\endcsname
    \csname the#1\endcsname
    \protect\@secnumpunct
  }%
}
\makeatother

\newcommand*{\hurl}  [2][www.]{\href{http://#1#2}{\nolinkurl{#2}}}
\newcommand*{\hemail}[1]{\href{mailto:#1}{\nolinkurl{#1}}}
\newcommand*{\DOI}   [1]{\href{http://dx.doi.org/#1}{\nolinkurl{#1}}}
\newcommand*{\arXiv} [1]{\href{http://www.arxiv.org/abs/#1}{\nolinkurl{arXiv:#1}}}
\newcommand*{\MSC}   [1]{\href{http://www.ams.org/msc/msc2010.html?t=#1}{#1}}

\setlist{leftmargin=*}
\setlist[enumerate]{label=(\roman*)}

\numberwithin{equation}{section}

\newcommand*{\definetheorem}[3][equation]{%
  \newaliascnt{#2}{#1}
  \newtheorem{#2}[#2]{#3}
  \aliascntresetthe{#2}
  \expandafter\def\csname #2autorefname\endcsname{#3}
}
\makeatletter
\newcommand*{\definetheoremsame}[2][equation]{%
  \definetheorem[#1]{\zap@space#2 \@empty}{\capitalisewords{#2}}
}
\makeatother

\theoremstyle{plain}
\forcsvlist{\definetheoremsame}%
{addendum, corollary, lemma, proposition, technical theorem, theorem}

\theoremstyle{definition}
\forcsvlist{\definetheoremsame}%
{definition, example, question, remark}

\newcommand*{\one}  {\text{\ding{192}}} 
\newcommand*{\two}  {\text{\ding{193}}} 
\newcommand*{\three}{\text{\ding{194}}} 
\newcommand*{\four} {\text{\ding{195}}} 
\newcommand*{\five} {\text{\ding{196}}}

\DeclareMathAlphabet{\matheurm}{U}{eur}{m}{n}

\newcommand*{\define}[5]{%
  \ifstrequal{#2}{*}{\expandafter#1\expandafter*}{\expandafter#1}%
  \csname#4#5\endcsname{#3{#5}}
}
\forcsvlist{\define{\newcommand}{*}{\mathbf}{B}}{%
  A,B,C,D,E,F,G,H,I,J,K,L,M,N,O,P,Q,R,S,T,U,V,W,X,Y,Z,%
}
\forcsvlist{\define{\newcommand}{*}{\mathcal}{C}}{%
  A,B,C,D,E,F,G,H,I,J,K,L,M,N,O,P,Q,R,S,T,U,V,W,X,Y,Z,%
}
\forcsvlist{\define{\newcommand}{*}{\mathbb}{I}}{%
  A,B,C,D,E,F,G,H,I,  K,L,M,N,O,P,Q,R,S,T,U,V,W,X,Y,Z%
}                

\forcsvlist{\define{\DeclareMathOperator}{}{}{}}%
{asbl, aut, coker, conj, fun, hocoeq, hocofib, hoeq, hofib, id, im, incl, Lan, map, mor, obj, pr, pt, tors}
\DeclareMathOperator{\ii}{in}

\forcsvlist{\define{\DeclareMathOperator}{*}{}{}}%
{colim, hocolim, holim}

\forcsvlist{\define{\newcommand}{*}{\matheurm}{}}%
{Cat, Groupoids, Or, Sets, Sp, Top}
\newcommand*{\RFcat}{{\matheurm{RF}}}

\forcsvlist{\define{\newcommand}{*}{\mathbf}{}}%
{K, THH, TR, TC, C, T, Wh}
\newcommand*{\tr}{\mathit{TR}}
\newcommand*{\tc}{\mathit{TC}}
\newcommand*{\wh}{\mathit{Wh}}

\newcommand*{\TO}  [1][]{\stackrel{#1}{\longrightarrow}}
\newcommand*{\FROM}[1][]{\stackrel{#1}{\longleftarrow}}
\newcommand*{\MOR} [4][]{#2\colon#3\TO[#1]#4}
\newcommand*{\AND}{\qquad\text{and}\qquad}

\DeclarePairedDelimiterX\SET[2]{\{}{\}}{\,#1\;\delimsize\vert\;#2\,}
\DeclarePairedDelimiter\real{\lvert}{\rvert}

\newcommand*{\ds}{\displaystyle}
\newcommand*{\ts}{\textstyle}

\newcommand*{\op}  {{\operatorname{op}}}

\newcommand*{\prlast}{\operatorname{pr}_{\operatorname{last}}}
\newcommand*{\prlasto}{\overline{\operatorname{pr}}_{\operatorname{last}}}

\DeclareMathOperator*{\smallcoprod}{\ts\coprod}
\DeclareMathOperator*{\smallprod}  {\ts\prod}
\DeclareMathOperator*{\tensor}     {\otimes}
\DeclareMathOperator*{\timesd}     {\times}
\DeclareMathOperator*{\sma}        {\wedge}
\newcommand*{\ssma}{\!\sma\!} 

\newcommand*{\D}{\text{-}} 

\newcommand*{\spec}[1]{\mathbb{#1}} 

\newcommand*{\sh}{\smash{\operatorname{sh}}\vphantom{\Omega}}

\newcommand*{\Cp}    [2][p]{C_{{#1}^{#2}}}

\newcommand*{\limone}[1][1]{\sideset{}{^{#1}}\lim}

\newcommand*{\Cyc} {{\mathcal{C} \hspace{-.2ex}\mathit{yc}}}
\newcommand*{\VCyc}{{\mathcal{VC}\hspace{-.2ex}\mathit{yc}}}
\newcommand*{\FCyc}{{\mathcal{FC}\hspace{-.2ex}\mathit{yc}}}
\newcommand*{\Fin} {{\mathcal{F} \hspace{-.2ex}\mathit{in}}}

\newcommand*{\oid}[2]{#1\!\smallint\!#2} 

\newcommand*{\BHM}{Bökstedt-Hsiang-Madsen}

\newcommand*{\cplus}[1][]{\ifstrequal{#1}{*}{C}{c}onnective\textsuperscript{+}}

\newcommand*{\MNM}[3]{(\mathit{#1M}_{#2\subset#3})}

\newcommand*{\Sym}{\mathit{\Sigma}}

\begin{document}

\begin{abstract}
We use assembly maps to study~$\TC(\spec{A}[G];p)$, the topological cyclic homology at a prime~$p$ of the group algebra of a discrete group~$G$ with coefficients in a connective ring spectrum~$\spec{A}$.
For any finite group, we prove that the assembly map for the family of cyclic subgroups is an isomorphism on homotopy groups.
For infinite groups, we establish pro-isomorphism, (split) injectivity, and rational injectivity results, as well as counterexamples to injectivity and surjectivity.
In particular, for hyperbolic groups and for virtually finitely generated abelian groups, we show that the assembly map for the family of virtually cyclic subgroups is injective but in general not surjective.
\end{abstract}


\maketitle
\tableofcontents
\thispagestyle{empty}


\section{Introduction}

The goal of this paper is to study topological cyclic homology of group algebras using assembly maps.
Since it was invented by Bökstedt, Hsiang, and Madsen in~\cite{BHM}, $\tc$ has been extensively studied, and deep structural and computational results have been established for example in \citelist{
\cite{H-p-typical}
\cite{HM-top}
\cite{McCarthy-rel}
\cite{Dundas-rel}
\cite{AR-top}
\cite{HM-annals}
\cite{GH}
\cite{BM-loc}
\cite{AGHL}
\cite{BGT}
\cite{BM-theory}
}, among many other works;
we~recommend the book~\cite{Dundas} for an overview.
New interactions with arithmetic and algebraic geometry, inspired by~\cite{BMS}, have recently amplified the interest in~$\tc$ and related theories.

Explicit computations usually rely on the connection to the de~Rham-Witt complex, and there are only few theorems about non-commutative rings, e.g., \citelist{
\cite{H-non-comm}
\cite{AR-Morava}
\cite{Angeltveit}
}.
Our results here show that assembly maps provide a powerful tool to study~$\tc$ of group algebras.

Consider a ring or more generally a connective ring spectrum~$\spec{A}$, and let $p$ be a prime.
Consider also a discrete group~$G$ and a family~$\CF$ of subgroups of~$G$.
As explained in \autoref{ASBL}, the assembly map for topological cyclic homology is a map of spectra
\begin{equation}
\label{eq:asbl-TC}
EG(\CF)_+\sma_{\Or G}\TC(\spec{A}[\oid{G}{-}];p)
\TO
\TC(\spec{A}[G];p)
\,,
\end{equation}
whose source may be interpreted as the homotopy colimit of~$\TC(\spec{A}[H];p)$ as $H$ ranges over the subgroups that belong to the family~$\CF$.
We always and tacitly assume that the symmetric ring spectrum~$\spec{A}$ is \cplus, in the sense of \autoref{cplus}.
This is a mild technical condition, which is satisfied by the sphere spectrum~$\spec{S}$ and by the Eilenberg-Mac Lane spectra of discrete rings.
We point out that all the results mentioned below hold not only for~$\tc$ but also for~$\tr$; see \autoref{also-TR}.

Our first result states that for any finite group~$G$ the assembly map~\eqref{eq:asbl-TC} for the family of cyclic subgroups induces isomorphisms on homotopy groups.
We should think of this result as an integral induction theorem for topological cyclic homology, in the spirit of Artin and Brauer induction in the representation theory of finite groups.

\begin{theorem}[isomorphism]
\label{iso-TC-finite-groups}
If the group~$G$ is finite, then the assembly map for the family of cyclic subgroups
\[
EG(\Cyc)_+\sma_{\Or G}\TC(\spec{A}[\oid{G}{-}];p)
\TO
\TC(\spec{A}[G];p)
\]
is a $\pi_*$-isomorphism.
\end{theorem}

\autoref{iso-TC-finite-groups} allows us to attack explicit computations for non-cyclic finite groups; more precisely, to reduce such computations to the cyclic subgroups.
We make this explicit in the smallest example: that of the symmetric group~$\Sym_3$.

\begin{proposition}
\label{TC(A[S3];p)}
For any prime~$p$ there is a $\pi_*$-isomorphism
\[
\TC(\spec{A}[C_2];p)
\vee
\widetilde\TC(\spec{A}[C_3];p)_{hC_2}
\TO[\simeq]
\TC(\spec{A}[\Sym_3];p)
\,,
\]
where $C_2$ acts on~$C_3$ by sending the generator to its inverse, and $\widetilde\TC(\spec{A}[G];p)$ is the homotopy cofiber of the map
\(
\TC(\spec{A};p)
\TO
\TC(\spec{A}[G];p)
\)
induced by the inclusion.
\end{proposition}

For infinite groups we obtain an analog of \autoref{iso-TC-finite-groups} for the pro-spectrum whose homotopy limit defines~$\TC$.

\begin{theorem}[pro-isomorphism]
\label{pro-iso-TC}
For any group~$G$, the assembly maps for the family of cyclic subgroups induce a strict map of pro-spectra
\begin{align*}
\Bigl\{
EG(\Cyc)_+\sma_{\Or G}\TC^n(\spec{A}[\oid{G}{-}];p)
\Bigr\}
&\TO
\Bigl\{
\TC^n(\spec{A}[G];p)
\Bigr\}
\\\shortintertext{which in each level of the pro-system is a $\pi_*$-isomorphism. In particular, there is a $\pi_*$-isomorphism}
\holim_{n\in\IN}\Bigl(EG(\Cyc)_+\sma_{\Or G}\TC^n(\spec{A}[\oid{G}{-}];p)\Bigr)
&\TO
\TC(\spec{A}[G];p)
\,.
\end{align*}
\end{theorem}

In order to deduce results about the assembly map~\eqref{eq:asbl-TC} from \autoref{pro-iso-TC}, one encounters the problem that the processes of forming homotopy limits and assembly maps do not commute in general.
For certain classes of groups, we show how this leads to results about injectivity but failure of surjectivity.

\begin{theorem}[injectivity]
\label{inj-TC-examples}
Assume that one of the following conditions holds:
\begin{enumerate}
\item
\label{i:Fin}
$\CF=\Fin$ is the family of finite subgroups, and there is a universal space $EG(\Fin)$ of finite type;
\item
\label{i:VCyc}
$\CF=\VCyc$ is the family of virtually cyclic subgroups, and $G$ is hyperbolic or virtually finitely generated abelian.
\end{enumerate}
Then the assembly map
\[
EG(\CF)_+\sma_{\Or G}\TC(\spec{A}[\oid{G}{-}];p)
\TO
\TC(\spec{A}[G];p)
\]
is $\pi_*$-injective.
Moreover, in case~\ref{i:Fin} the assembly map is split injective.
\end{theorem}

Here we say that a map of spectra~$\MOR{f}{\BX}{\BY}$ is split injective if there is a map $\MOR{g}{\BY}{\BZ}$ to some other spectrum~$\BZ$ such that $g\circ f$ is a $\pi_*$-isomorphism.
We emphasize that the existence of a universal space~$EG(\Fin)$ of finite type is a mild condition.
For example, all the following groups have even finite (i.e., finite type and finite dimensional) models for~$EG(\Fin)$:
hyperbolic groups;
CAT(0)~groups;
cocompact lattices in virtually connected Lie groups;
arithmetic groups in semisimple connected linear $\IQ$-algebraic groups;
mapping class groups;
outer automorphism groups of free groups.
For more information we refer to \cite{L-survey}*{Section~4} and also to~\cite{Mislin} in the case of mapping class groups.

In fact, \autoref{inj-TC-examples} is a special case of a more general result, \autoref{inj-TC-technical} below.
But first we want to highlight the following negative result about surjectivity.

\begin{theorem}[failure of surjectivity]
\label{not-surj-TC}
The assembly map for the family of virtually cyclic subgroups
\[
EG(\VCyc)_+\sma_{\Or G}\TC(\spec{A}[\oid{G}{-}];p)
\TO
\TC(\spec{A}[G];p)
\]
is not always $\pi_*$-surjective.
For example, it is not surjective on~$\pi_{-1}$ if
$\spec{A}=\IZ_{(p)}$ 
and $G$ is either finitely generated free abelian or torsion-free hyperbolic, but not cyclic.
\end{theorem}

This result is in strong contrast to the situation in algebraic $K$-theory and $L$-theory, where the Farrell-Jones Conjecture~\cite{FJ-iso} predicts that the assembly map for the family of virtually cyclic subgroups is a $\pi_*$-isomorphism for any group and coefficients in any discrete ring or the sphere spectrum.
While still open in general, the Farrell-Jones Conjecture has been verified for an astonishingly large class of groups, including in particular the groups in~\autoref{not-surj-TC}; see \citelist{\cite{LR-survey} \cite{L-ICM} \cite{RV-survey}} for more information.
We do not believe that by enlarging the family one could obtain a version of such a conjecture that has a chance to be true for~$\tc$.

It is also interesting to notice that, at the time of writing, the analog of \autoref{inj-TC-examples}\ref{i:Fin} in algebraic $K$-theory is not known.
In particular, it is not known whether the $K$-theory assembly map for the family of finite subgroups is split injective in the case of outer automorphism groups of free groups---all we know is that, in the case when $\spec{A}=\IZ$ or~$\IS$, it is eventually rationally injective~\cite{kc}*{Theorem~1.15, page~936}.

Next, we introduce some terminology needed to formulate our more technical results below.

\begin{definition}[$p$-radicable]
\label{p-radicable}
A family $\CF$ of subgroups of~$G$ is called \emph{$p$\nobreakdash-radicable} provided that, for every $g\in G$, we have $\langle g\rangle\in\CF$ if and only if $\langle g^p\rangle\in\CF$.
\end{definition}

Since, by definition, any family~$\CF$ is closed under passage to subgroups, $\langle g\rangle\in\CF$ always implies that $\langle g^p\rangle\in\CF$, but the converse is not necessarily true.
Notice, for example, that the trivial family~$1$ is $p$-radicable if and only if there are no elements of order~$p$ in~$G$.

\begin{technicaltheorem}
\label{inj-TC-technical}
\leavevmode
\begin{enumerate}
\item
\label{i:split-inj-TC-technical}
Assume that~$\CF$ is $p$-radicable and has a universal space $EG(\CF)$ of finite type.
Then the assembly map~\eqref{eq:asbl-TC} is split injective.
\item
\label{i:pi*-inj-TC-technical}
Assume that~$\CF$ can be written as a directed union
\[
\CF=\bigcup_{j\in\CJ}\CF_j
\]
of subfamilies~$\CF_j$, each of which is $p$-radicable and has a universal space $EG(\CF_j)$ of finite type.
Then the assembly map~\eqref{eq:asbl-TC} is $\pi_*$-injective.
\end{enumerate}
\end{technicaltheorem}

Since the family of finite subgroups is $p$-radicable, \autoref{inj-TC-examples}\ref{i:Fin} is just a special case of \autoref{inj-TC-technical}\ref{i:split-inj-TC-technical}.
In \autoref{EG(VCyc)} we show that the groups in \autoref{inj-TC-examples}\ref{i:VCyc} satisfy the assumption of \autoref{inj-TC-technical}\ref{i:pi*-inj-TC-technical} for $\CF=\VCyc$.

Finally, we establish the following rational result.
Here we say that an abelian group~$M$ is \emph{almost finitely generated} if its torsion subgroup~$\tors M$ is annihilated by some~$r\in\IZ-\{0\}$, and the quotient~$M/\tors M$ is finitely generated.
We say that a map of spectra~$\MOR{f}{\BX}{\BY}$ is $\pi_n^\IQ$-injective if $\pi_n(f)\tensor_\IZ\IQ$ is injective.

\begin{theorem}[rational injectivity]
\label{inj-TC-rationally}
Let $N\geq0$ be an integer.
Assume that:
\begin{enumerate}[label=(\alph*)]
\item\label{i:p-radicable}
$\CF$ is $p$-radicable;
\item\label{i:F-in-Fin}
$\CF$ contains only finite subgroups;
\item\label{i:(F)-finite}
$\CF$ contains only finitely many conjugacy classes of subgroups;
\item\label{techA}
for every~$H\in\CF$ and for every $1\leq s\leq N+2$, the integral group homology $H_s(BZ_GH;\IZ)$ of the centralizer of~$H$ in~$G$ is an almost finitely generated abelian group.
\end{enumerate}
Then the assembly map~\eqref{eq:asbl-TC} is $\pi_n^\IQ$-injective for all~$-\infty<n\le N$.
\end{theorem}

\begin{remark}[failure of injectivity without finiteness assumptions]
Without assumption~\ref{techA} \autoref{inj-TC-rationally} would be false.
As a counterexample, consider the additive group of the rational numbers~$G=\IQ$ and the trivial family $\CF=1$.
Obviously, assumptions~\ref{i:p-radicable}, \ref{i:F-in-Fin}, and~\ref{i:(F)-finite} are satisfied, but~\ref{techA} is not.
The map~\eqref{eq:asbl-TC} becomes
\[
B\IQ_+\sma\TC(\spec{A};p)
\TO
\TC(\spec{A}[\IQ];p)
\,,
\]
and in~\cite{kc}*{Remark~3.7, pages~946--947} we show that this is \emph{not} $\pi_n^\IQ$-injective when~$\spec{A}$ is the sphere spectrum~$\spec{S}$.
\end{remark}

\begin{remark}[Thompson's group~$T$]
Let $\CF=\Fin$.
Then all the assumptions of \autoref{inj-TC-rationally} are satisfied if there is a universal space~$EG(\Fin)$ of finite type, but not vice versa; see e.g.~\cite{kc}*{Proposition~2.1, page~939}.
In~\cite{GV} it is proved that Thompson's group~$T$ of orientation preserving, piecewise linear, dyadic homeomorphisms of the circle satisfies~\ref{techA} for all~$N$, but not~\ref{i:(F)-finite}.
It is an interesting open question whether the assembly map~\eqref{eq:asbl-TC} is $\pi_n^\IQ$-injective for Thompson's group~$T$.
Notice that the main result of~\cite{kc}*{Theorem~1.13, page~935} gives an affirmative answer to the analogous question for the assembly map in connective algebraic $K$-theory when $\spec{A}=\IZ$ or~$\IS$, provided that a weak version of the Leopoldt-Schneider Conjecture in algebraic number theory holds for all cyclotomic fields.
\end{remark}

\begin{remark}[\BHM's functor~$C$]
The analog of \autoref{inj-TC-rationally} for the assembly map
\[
EG(\CF)_+\sma_{\Or G}\C(\spec{A}[\oid{G}{-}];p)
\TO
\C(\spec{A}[G];p)
\]
for \BHM's functor~$C$, a variant of~$\tc$, is proved in~\cite{kc}*{Theorem~1.19(ii), pages~937--938, Theorem~9.5, page~978}.
For~$C$ we only need to assume conditions \ref{i:F-in-Fin} and~\ref{techA}, and neither \ref{i:p-radicable} nor~\ref{i:(F)-finite} are required.
Moreover, assumption~\ref{techA} may be weakened by replacing $N+2$ with~$N+1$.
\end{remark}

\begin{addendum}[$\tr$]
\label{also-TR}
All results mentioned in this introduction (Theorems \ref{iso-TC-finite-groups}, \ref{pro-iso-TC}, \ref{inj-TC-examples}, \ref{not-surj-TC}, \ref{inj-TC-technical}, \ref{inj-TC-rationally}, and \autoref{TC(A[S3];p)}) hold also for~$\TR(-;p)$ instead of~$\TC(-;p)$.
For~$\TR(-;p)$, in \autoref{not-surj-TC} replace $\pi_{-1}$ with~$\pi_0$ (see \autoref{not-surj-TR-TC}), and in \autoref{inj-TC-rationally} assumption~\ref{techA} can be weakened by replacing $N+2$ with~$N+1$.
\end{addendum}


\subsection*{Acknowledgments}
We thank the referee for a detailed and thoughtful report.
We also thank the Hausdorff Research Institute for Mathematics in Bonn, where parts of this work were completed during the 2015 Trimester Program on ``Homotopy theory, manifolds, and field theories.''
We have been financially supported
by the first author's Leibniz Award, granted by the Deutsche Forschungsgemeinschaft,
and by his European Research Council Advanced Grant ``KL2MG-interactions'' (\#662400);
by the Collaborative Research Center~647 in Berlin;
and
by a grant from the Simons Foundation (\#419561, Marco Varisco).


\section{Preliminaries}

In this section we fix our notation and terminology, and we recall the definition of assembly maps.

\subsection{Spaces and spectra}
We work in the category~$\Top$ of compactly generated and weak Hausdorff spaces, which from now on we simply call spaces.
We let $\CT$ be the category of pointed spaces, and we write $\CW$ for the full subcategory of pointed spaces homeomorphic to countable CW complexes.
Given a group~$G$, we denote by $\Top^G$ the category of left $G$-spaces and $G$-equivariant maps; the discrete and the base pointed versions of this category are denoted $\Sets^G$ and~$\CT^G$, respectively.
A~finite cyclic group of order~$n$ is denoted~$C_n$.

Following~\cite{kc}*{Subsection~4I, page~951}, we write~$\IN\Sp$, $\Sigma\Sp$, and~$\CW\CT$ for the categories of naive spectra, symmetric spectra, and $\CW$-spaces, respectively.
All our spectra are defined using sequences of spaces, as opposed to simplicial sets.
We also consider the category~$\CW\CT^{S^1}$ of $S^1$-equivariant $\CW$-spaces, as our model of topological Hochschild homology naturally lives there.
We denote by $\Sigma\Sp\D\Cat$ the category of small symmetric spectral categories, i.e., categories enriched over the symmetric monoidal category~$\Sigma\Sp$.
We also need the following technical definition of \cplus{} symmetric ring spectra, which appears as a hypothesis in some of our results.

\begin{definition}
We say that a symmetric spectrum $\spec{E}$ is:
\begin{enumerate}

\item\emph{strictly connective}
if for every~$x\geq0$ the space $\spec{E}_x$ is $(x-1)$-connected;

\item\emph{convergent}
if there exists a non-decreasing function $\MOR{\lambda}{\IN}{\IZ}$ such that $\lim_{x\to\infty}\lambda(x)=\infty$ and the adjoint structure map $\spec{E}_x\TO\Omega\spec{E}_{x+1}$ is $(x+\lambda(x))$-connected for every~$x\geq0$;

\item\emph{well pointed}
if for every $x\in\IN$ the space~$\spec{E}_x$ is well pointed.

\end{enumerate}
\end{definition}

\begin{definition}[{\cplus[*]}]
\label{cplus}
We say that a symmetric ring spectrum~$\spec{A}$ is \emph{\cplus} if it is strictly connective, convergent, well pointed, and the unit map~$\spec{S}\TO\spec{A}$ induces a cofibration~$S^0\TO\spec{A}_0$.
\end{definition}

Any $(-1)$-connected $\Omega$-spectrum is strictly connective and convergent.
The sphere spectrum~$\spec{S}$ and suitable models for all Eilenberg-Mac Lane ring spectra of discrete rings are \cplus.

\subsection{$\THH$, $\TR$, and~$\TC$}
Given a symmetric ring spectrum~$\spec{A}$ or, more generally, a symmetric spectral category~$\spec{D}$, topological Hochschild homology defines an $S^1$-equivariant $\CW$-space~$\THH(\spec{D})$.
Some details of the construction are recalled in \autoref{PRO-ISO}; for more information, we refer to~\cite{kc}*{Section~6} for the specific model that we use, and to~\cite{Dundas} in general.

Fix a prime~$p$.
As $n\ge1$ varies, the $\Cp{n}$-fixed points of~$\THH(\spec{D})$ are related by maps
\[
\MOR{R,F}{\THH(\spec{D})^{\Cp{n}}}{\THH(\spec{D})^{\Cp{n-1}}}
\]
satisfying~$RF=FR$.
The Frobenius map~$F$ is the inclusion of fixed points;
the restriction map~$R$ is more complicated, and the essential ingredient for its construction is reviewed in \autoref{PRO-ISO}.

Following e.g.~\cite{H-survey}, we write
\begin{align*}
\TR^n(\spec{D};p)
&=
\THH(\spec{D})^{\Cp{n-1}}
,
\\
\TC^n(\spec{D};p)
&=
\hoeq\Bigl(
\MOR{R,F}{\THH(\spec{D})^{\Cp{n-1}}}{\THH(\spec{D})^{\Cp{n-2}}}
\Bigr)
.
\end{align*}
The maps
$\MOR{R}{\TR^{n+1}(\spec{D};p)}{\TR^{n}(\spec{D};p)}$
induce
$\MOR{R}{\TC^{n+1}(\spec{D};p)}{\TC^{n}(\spec{D};p)}$,
and we define
\[
\TR(\spec{D};p)=\holim_{n\in\IN}\TR^n(\spec{D};p)
\AND
\TC(\spec{D};p)=\holim_{n\in\IN}\TC^n(\spec{D};p)
\]
with the homotopy limit taken over the maps~$R$ above.
Equivalently up to $\pi_*$-isomorphism, one can define~$\TC(\spec{D};p)$ as
\[
\holim_{n\in\obj\RFcat}\THH(\spec{D})^{C_{p^n}}
,
\]
where $\RFcat$ is the category with set of objects~$\IN$, and where the morphisms from $m$ to~$n$ are the pairs $(i,j)\in\IN\times\IN$ with $i+j=m-n$.

Now fix a symmetric ring spectrum~$\spec{A}$.
Given a group~$G$ or more generally a groupoid~$\CG$, consider the symmetric spectral category~$\spec{A}[\CG]$ with set of objects $\obj\CG$ and morphism spectra $\spec{A}\sma\CG(x,y)_+$.
We thus get functors
\[
\THH(\spec{A}[-])
\,,\
\TR^n(\spec{A}[-];p)
\,,\
\TC^n(\spec{A}[-];p)
\,,\
\TR(\spec{A}[-];p)
\,,\
\TC(\spec{A}[-];p)
\]
from the category of groupoids to the category of naive spectra, and all these functors send equivalences of groupoids to $\pi_*$-isomorphisms.

\subsection{Assembly maps}
\label{ASBL}
Following the approach of~\cite{Davis-L}, the input for the construction of assembly maps is for us a functor
\[
\MOR{\T}{\Groupoids}{\IN\Sp}
\]
that preserves equivalences, i.e., that sends equivalences of groupoids to $\pi_*$\=/isomorphisms.
Given a group~$G$, consider the functor~$\MOR{\oid{G}{-}}{\Sets^G}{\Groupoids}$ that sends a $G$-set~$S$ to its action groupoid~$\oid{G}{S}$, with $\obj\oid{G}{S}=S$ and $\mor_{\oid{G}{S}}(s,s')=\SET{g\in G}{gs=s'}$.
Restricting to the orbit category~$\Or G$, i.e., the full subcategory of~$\Sets^G$ with objects~$G/H$ as $H$ varies among the subgroups of~$G$, we obtain the composition
\[
\begin{tikzcd}[cramped]
\Or G
\arrow[r, hook]
&
\Sets^G
\arrow[r, "\oid{G}{-}"]
&
\Groupoids
\arrow[r, "\T"]
&
\IN\Sp
\,.
\end{tikzcd}
\]
Finally, we take the left Kan extension of $\T(\oid{G}{-})$ along the full and faithful functor~$\iota\colon\Or{G}\hookrightarrow\Top^G$.
Given a $G$-space~$X$,
\[
\bigl(\Lan_\iota\T(\oid{G}{-})\bigr)
(X)
=
\smash{X_+\sma_{\Or G}\T(\oid{G}{-})}
\]
is defined as the coend of the functor
\begin{align*}
(\Or G)^\op\times\Or G
&\TO
\IN\Sp
\,,
\\
(G/H,G/K)
&\longmapsto
\map(G/H,X)^G_+\sma\T(\oid{G}{\,(G/K)})\cong X^H_+\sma\T(\oid{G}{\,(G/K)})
\,.
\end{align*}
Notice that $\pt_+\sma_{\Or G}\T(\oid{G}{-})\cong\T(G)$ and, for any subgroup~$H\leq G$, the fact that $\T$ preserves equivalences implies that $G/H_+\sma_{\Or G}\T(\oid{G}{-})$ is $\pi_*$-isomorphic to~$\T(H)$.

Now consider a family~$\CF$ of subgroups of~$G$ (i.e., a collection of subgroups closed under passage to subgroups and conjugates)
and consider a universal $G$-space $EG(\CF)$.
This is a $G$-CW~complex characterized up to $G$-homotopy equivalence by the property that, for any subgroup~$H\le G$, the $H$-fixed point space
\[
\bigl(EG(\CF)\bigr)^H
\,
\text{ is }
\begin{cases}
\text{empty}&\text{if $H\not\in\CF$;}\\
\text{contractible}&\text{if $H\in\CF$.}
\end{cases}
\]
The projection $EG(\CF)\TO\pt=G/G$ induces then the \emph{assembly map}
\[
\MOR{\asbl}{EG(\CF)_+\sma_{\Or G}\T(\oid{G}{-})}{\T(G)}
\,.
\]

Finally, we remark that the source of the assembly map is a model for
\[
\hocolim_{\substack{G/H\in\obj\Or G\\\text{s.t. }H\in\CF}}\T(\oid{G}{\,(G/H)})
\,,
\]
the homotopy colimit of the restriction of~$\T(\oid{G}{-})$ to the full subcategory of~$\Or G$ spanned by the objects $G/H$ with $H\in\CF$.


\section{Pro-isomorphism result}
\label{PRO-ISO}

In this section we prove \autoref{pro-iso-TC} and an important intermediate step, \autoref{split-THH-fix}, on which our other isomorphism and injectivity results are based.
The starting point is the following theorem from~\cite{kc}.

\begin{theorem}
\label{split-THH}
Let $G$ be a group, $\CF$ a family of subgroups of~$G$, $\spec{A}$ a symmetric ring spectrum, and $C$ a finite subgroup of~$S^1$.
\begin{enumerate}
\item
\label{i:split-THH}
Consider the following commutative diagram in~$\CW\CT^{S^1}$.
\[
\begin{tikzcd}[column sep=large]
\ds EG(\CF)_+\sma_{\Or G}\THH    (\spec{A}[\oid{G}{-}])
\arrow[r, "\asbl"]
\arrow[d, "\id\sma\pr_\CF"']
&
\THH    (\spec{A}[G])
\arrow[d, "\pr_\CF"]
\\
\ds EG(\CF)_+\sma_{\Or G}\THH_\CF(\spec{A}[\oid{G}{-}])
\arrow[r, "\asbl"']
&
\THH_\CF(\spec{A}[G])
\end{tikzcd}
\]
The left-hand map and the bottom map are $\pi_*$-isomorphisms of the underlying non-equivariant spectra.
If $\CF$ contains all cyclic groups, then the right-hand map is an isomorphism.
\item
\label{i:split-THH-sh-fix}
Assume that $\spec{A}$ is \cplus.
Consider the following commutative diagram in~$\CW\CT$.
\[
\hspace{2em}
\begin{tikzcd}[column sep=large]
\ds EG(\CF)_+\sma_{\Or G}\bigl(\sh^{ES^1_+}\THH    (\spec{A}[\oid{G}{-}])\bigr)^C
\arrow[r, "\asbl"]
\arrow[d, "\id\sma\pr_\CF"']
&
\bigl(\sh^{ES^1_+}\THH    (\spec{A}[G])\bigr)^C
\arrow[d, "\pr_\CF"]
\\
\ds EG(\CF)_+\sma_{\Or G}\bigl(\sh^{ES^1_+}\THH_\CF(\spec{A}[\oid{G}{-}])\bigr)^C
\arrow[r, "\asbl"']
&
\bigl(\sh^{ES^1_+}\THH_\CF(\spec{A}[G])\bigr)^C
\end{tikzcd}
\]
The left-hand map and the bottom map are $\pi_*$-isomorphisms.
If $\CF$ contains all cyclic groups, then the right-hand map is an isomorphism.
\end{enumerate}
\end{theorem}

\begin{proof}
\ref{i:split-THH} is \cite{kc}*{Theorem~6.1, page~960}
and
\ref{i:split-THH-sh-fix} is \cite{kc}*{Corollary~8.4, pages~976--977}.
\end{proof}

We recall that the key to deduce \ref{i:split-THH-sh-fix} from~\ref{i:split-THH} is a natural zig-zag of $\pi_*$\=/isomorphisms, based on~\cite{RV}, between the underlying orthogonal spectra of
\[
\bigl(\THH_\CF(\spec{A}[\oid{G}{S}])\bigr)_{hC}
\AND
\bigl(\sh^{ES^1_+}\THH_\CF(\spec{A}[\oid{G}{S}])\bigr)^{C}
\]
that commutes with~$\pr_\CF$.
The technical definition of the equivariant shift~$\sh^{ES^1_+}\BX$ of an $S^1$-$\CW$-space~$\BX$ is given in~\cite{kc}*{Definition~4.10 and~4L, pages~953--954}.

We also review the definition of~$\THH_\CF$ and~$\pr_\CF$, since we need them explicitly in the proof of \autoref{split-THH-fix} below.
We begin with~$\THH$ itself, following the notation of~\cite{kc}*{Section~6, pages~960--961}.

Let $\CI$ be the category with objects the finite sets $x=\{1,2,\dotsc,x\}$ for all~$x\ge0$, and morphisms all the injective functions, and let $\CI^{[q]}=\CI^{q+1}$ for each~$q\in\IN$.
Given a symmetric spectral category~$\spec{D}$, there are functors
\begin{align*}
\MOR{cn_{[q]}\spec{D}}{\CI^{[q]}}{{}&\ \CT}
\,,
\\
\vec{x}=(x_0,\dotsc,x_q)
\longmapsto{}
&
\quad\bigvee_{\mathclap{\substack{d_0,\dotsc,d_q\\\text{in}\,\obj\spec{D}}}}\quad
\spec{D}(d_0,d_q)_{x_0}\sma
\spec{D}(d_1,d_0)_{x_1}\sma
\dotsb\sma
\spec{D}(d_q,d_{q-1})_{x_q}
\,;
\\[\medskipamount]
\MOR{ M_{[q]}\spec{D}}{\CI^{[q]}}{{}&\CW\CT}
\,,
\\
\vec{x}=(x_0,\dotsc,x_q)
\longmapsto{}
&
\map\bigl(cn_{[q]}\spec{S}(\vec{x}),\;-\sma cn_{[q]}\spec{D}(\vec{x})\bigr)=
\\
&=
\map\bigl(S^{x_0}\sma\dotsb\sma S^{x_q},\;-\sma cn_{[q]}\spec{D}(x_0,\dotsc,x_q)\bigr)
,
\end{align*}
where $-$ denotes the variable in~$\CW$.
Then $THH_{[q]}(\spec{D})$ is defined as the homotopy colimit of~$M_{[q]}\spec{D}$, and after geometric realization we obtain the functor
\[
\MOR{\THH(-)=\real{THH_\bullet(-)}}{\Sigma\Sp\D\Cat}{\CW\CT^{S^1}}
.
\]

Now we specialize to~$\spec{D}=\spec{A}[\oid{G}{S}]$ for a $G$-set~$S$.
Let $CN_\bullet(\oid{G}{S})$ be the cyclic nerve (compare e.g.~\cite{kc}*{Definition~5.1, page~956}), and let $\conj G$ be the set of conjugacy classes~$[g]$ of elements $g\in G$.
There are natural isomorphisms (see~\cite{kc}*{Lemma~6.4, page~962, and pages~957--958})
\[
cn_\bullet(\spec{A}[\oid{G}{S}])
\cong
\bigl(cn_\bullet\spec{A}\bigr)\sma CN_\bullet(\oid{G}{S})_+
\cong
\bigvee_{[c]\in\conj G}
\bigl(cn_\bullet\spec{A}\bigr)\sma CN_{\bullet,[c]}(\oid{G}{S})_+
\,,
\]
where $CN_{\bullet,[c]}(\oid{G}{S})$ denotes the preimage of~$[c]$ under the map of cyclic sets
\begin{align}
\label{eq:maptoconjG}
CN_{\bullet}(\oid{G}{S}) &\TO \conj G
\\[-2.5ex]
\raisebox{2ex}{\begin{tikzcd}[ampersand replacement=\&]
s_0
\arrow[rrrr, rounded corners,
       to path={    (\tikztostart.south)
                 |- +(3,-.3) [at end]\tikztonodes
                 -| (\tikztotarget.south)},
      "g_0"']
\&
s_1
\arrow[l, "g_1"']
\&
\dotsb
\arrow[l, "g_2"']
\&
s_{q-1}
\arrow[l, "g_{q-1}"']
\&
s_q
\arrow[l, "g_q"']
\end{tikzcd}}
&
\longmapsto
[g_0 g_1 \dotsb g_q ]
\,.
\hspace{8em}
\end{align}
Here $\conj G$ is viewed as a constant cyclic set.
The preimage of
\[
\conj_{\CF} G = \SET{[c]}{ \langle c \rangle \in \CF } \subseteq \conj G
\]
under~\eqref{eq:maptoconjG} is denoted~$CN_{\bullet,\CF}(\oid{G}{S})$.

Finally, the $\CF$-parts of~$\THH$ are defined as
\[
\smash{\THH_\CF(\spec{A}[\oid{G}{S}])
=
\real[\big]{THH_{\bullet,\CF}(\spec{A}[\oid{G}{S}])}
=
\real[\Big]{\hocolim_{\CI^\bullet}M_{\bullet,\CF}(\spec{A}[\oid{G}{S}])}
\,,}
\]
where
\begin{align*}
\MOR{M_{[q],\CF}(\spec{A}[\oid{G}{S}])}{\CI^{[q]}&}{\CW\CT}
\,,
\\
\vec{x}
&
\longmapsto
\map\bigl(cn_{[q]}\spec{S}(\vec{x}),\;-\sma cn_{[q]}\spec{A}(\vec{x})\sma CN_{[q],\CF}(\oid{G}{S})_+\bigr)
\,.
\end{align*}
Notice that there are maps
\[
\begin{tikzcd}
CN_{\bullet}    (\oid{G}{S})_+
\arrow[r, shift left, "\pr_{\CF}"]
&
CN_{\bullet,\CF}(\oid{G}{S})_+
\arrow[l, shift left, "\ii_{\CF}"]
\text{\,,\quad}
\THH    (\spec{A}[\oid{G}{S}])
\arrow[r, shift left, "\pr_{\CF}"]
&
\THH_\CF(\spec{A}[\oid{G}{S}])
\arrow[l, shift left, "\ii_{\CF}"]
\end{tikzcd}
\]
satisfying $\pr_{\CF}\circ\ii_{\CF}=\id$.

The main new result here is the following.

\begin{theorem}
\label{split-THH-fix}
Assume that the family~$\CF$ is $p$-radicable; see \autoref{p-radicable}.
Assume that the symmetric ring spectrum~$\spec{A}$ is \cplus.
For each~$n\geq0$ consider the following commutative diagram in~$\CW\CT$.
\[
\hspace{2em}
\begin{tikzcd}[column sep=large]
\ds EG(\CF)_+\sma_{\Or G}\bigl(\THH    (\spec{A}[\oid{G}{-}])\bigr)^{\Cp{n}}
\arrow[r, "\asbl"]
\arrow[d, "\id\sma\pr_\CF"']
&
\bigl(\THH    (\spec{A}[G])\bigr)^{\Cp{n}}
\arrow[d, "\pr_\CF"]
\\
\ds EG(\CF)_+\sma_{\Or G}\bigl(\THH_\CF(\spec{A}[\oid{G}{-}])\bigr)^{\Cp{n}}
\arrow[r, "\asbl"']
&
\bigl(\THH_\CF(\spec{A}[G])\bigr)^{\Cp{n}}
\end{tikzcd}
\]
The left-hand map and the bottom map are $\pi_*$-isomorphisms.
If $\CF$ contains all cyclic groups, then the right-hand map is an isomorphism.
\end{theorem}

\begin{proof}
We want to use the natural stable homotopy fibration sequence
\begin{equation}
\label{eq:hobif(R)}
\bigl(\sh^{ES^1_+}\THH(\spec{A}[\oid{G}{S}])\bigr)^{C_{p^{n  }}}
\xrightarrow{\pr_*}
                  \THH(\spec{A}[\oid{G}{S}])^{C_{p^{n  }}}
\TO[R]
                  \THH(\spec{A}[\oid{G}{S}])^{C_{p^{n-1}}}
\end{equation}
from~\cite{kc}*{Theorem~7.2, page~969} in order to deduce this result inductively  from~\autoref{split-THH}.
The problem is that the map
\begin{equation}
\label{eq:R}
\MOR{R}{\THH(\spec{A}[\oid{G}{S}])^{\Cp{n}}}{\THH(\spec{A}[\oid{G}{S}])^{\Cp{n-1}}}
\end{equation}
does not restrict in general to the $\CF$-parts of~$\THH$; see~\cite{kc}*{Warning~7.13, page~971}.
We explain how the assumption that $\CF$ is $p$-radicable solves this problem.

The key ingredient for the definition of the map~$R$ in~\eqref{eq:R} is given by the following two maps; compare~\cite{kc}*{page~970}.
\begin{equation}
\label{eq:res-to-fix}
\begin{tikzcd}[row sep=scriptsize]
\map\Bigl(cn_{p^{n}[q]}\spec{S}(p^{n}\vec{x}),\;-\sma cn_{p^{n}[q]}\spec{A}(p^{n}\vec{x})\sma CN_{p^{n}[q]}(\oid{G}{S})_+\Bigr)^{C_{p^{n}}}
\arrow[d, "\ f\mapsto f^{C_p}", shorten <=-1ex, shorten >=-1ex]
\\
\map\Bigl(\bigl(cn_{p^{n}[q]}\spec{S}(p^{n}\vec{x})\bigr)^{C_p},\;-\sma \bigl(cn_{p^{n}[q]}\spec{A}(p^{n}\vec{x})\bigr)^{C_p}\sma\bigl(CN_{p^{n}[q]}(\oid{G}{S})\bigr)^{C_p}_+\Bigr)^{C_{p^{n}}/C_p}
\\
\map\Bigl(cn_{p^{n-1}[q]}\spec{S}(p^{n-1}\vec{x}),\;-\sma cn_{p^{n-1}[q]}\spec{A}(p^{n-1}\vec{x})\sma CN_{p^{n-1}[q]}(\oid{G}{S})_+\Bigr)^{C_{p^{n-1}}}
\arrow[u, "\cong\ ", "\ \Delta"', shorten <=-1ex, shorten >=-1ex]
\end{tikzcd}
\end{equation}
Here $-$ denotes the variable in~$\CW$, and $p^n[q]$ denotes the concatenation of~$[q]$ with itself $p^n$ times, arising from the $p^n$-fold edgewise subdivision.
Note that $p^n[q]=[p^n(q+1)-1]$. 

The first map in~\eqref{eq:res-to-fix} is given by restricting to the $C_p$-fixed points, using the fact that the action on~$-$ is trivial.
For the second map, notice that there is a $C_{p^n}/C_p$-equivariant bijection
\begin{align}
\label{eq:Delta-map-CN}
\MOR[\cong]{\Delta_{\oid{G}{S}}}%
{CN_{p^{n-1}[q]}(\oid{G}{S})
&}{\bigl(CN_{p^{n}[q]}(\oid{G}{S})\bigr)^{C_p}}
,
\\
(g_0,\dotsc,g_r)&\longmapsto(g_0,\dotsc,g_r,g_0,\dotsc,g_r,\dotsc,g_0,\dotsc,g_r)
\,,
\\\shortintertext{where $r=p^{n-1}(q+1)-1$, as well as $C_{p^n}/C_p$-equivariant homeomorphisms}
\MOR[\cong]{\Delta_{\spec{A}}}%
{cn_{p^{n-1}[q]}\spec{A}(p^{n-1}\vec{x})
&}{\bigl(cn_{p^{n}[q]}\spec{A}(p^{n}\vec{x})\bigr)^{C_p}}
.
\\\shortintertext{In the case $\spec{A}=\spec{S}$, we have}
\MOR[\cong]{\Delta_{\spec{S}}}%
{(S^{x_0} \wedge \dots \wedge S^{x_{q}})^{\wedge p^{n-1}}
&}{((S^{x_0} \wedge \dots \wedge S^{x_{q}})^{\wedge p^n})^{C_p}}
.
\end{align}
The homeomorphisms $\Delta_{\spec{S}}$, $\Delta_{\spec{D}}$, and~$\Delta_{\oid{G}{S}}$, together with the obvious identification $C_{p^n}/C_p \cong C_{p^{n-1}}$ that takes every element of order~$p^n$ on the circle to its $p$-th power, induce the homeomorphism~$\Delta$ in~\eqref{eq:res-to-fix}.
Finally, the key map needed for the definition of~\eqref{eq:R} is defined as the composition of the top map with the inverse of the bottom map in~\eqref{eq:res-to-fix}.

The crucial observation now is that, given an arbitrary family~$\CF$, the bijection~$\Delta_{\oid{G}{S}}$ in~\eqref{eq:Delta-map-CN} induces a bijection
\[
\MOR[\cong]{\Delta_{\oid{G}{S},\CF}}%
{CN_{p^{n-1}[q],\!\sqrt[p]{\CF}}(\oid{G}{S})}%
{\bigl(CN_{p^{n}[q],\CF}(\oid{G}{S})\bigr)^{C_p}}
,
\]
where $CN_{\bullet,\!\sqrt[p]{\CF}}(\oid{G}{S})$ is defined as the preimage under~\eqref{eq:maptoconjG} of
\[
\SET{[c]}{ \langle c^p \rangle \in \CF } \subseteq \conj G
\,.
\]
We always have that $CN_{\bullet,{\CF}\vphantom{\sqrt[p]{\CF}}}(\oid{G}{S})\subseteq CN_{\bullet,\!\sqrt[p]{\CF}}(\oid{G}{S})$, but equality holds if and only if the family $\CF$ is $p$-radicable.
Therefore we see that the map~$R$ restricts to a map
\[
\MOR{R_\CF}{\THH_\CF(\spec{A}[\oid{G}{S}])^{\Cp{n}}}{\THH_\CF(\spec{A}[\oid{G}{S}])^{\Cp{n-1}}}
\]
if $\CF$ is $p$-radicable.

Now assume that this is the case.
Since $\Delta_{\oid{G}{S},\CF}\circ\pr_\CF=\pr_\CF\circ\Delta_{\oid{G}{S}}$ we obtain that $R_\CF\circ\pr_\CF=\pr_\CF\circ R$, and analogously for~$\ii_\CF$.
Also the natural map~$\pr_*$ in~\eqref{eq:hobif(R)} commutes with $\pr_\CF$ and~$\ii_\CF$.
Then, since retracts preserve stable homotopy fibration sequences, from~\eqref{eq:hobif(R)} we obtain for each~$n\geq1$ a stable homotopy fibration sequence in~$\CW\CT$
\[
\bigl(\sh^{ES^1_+}\THH_\CF(\spec{A}[\oid{G}{S}])\bigr)^{C_{p^{n  }}}
\xrightarrow{\pr_*}
                  \THH_\CF(\spec{A}[\oid{G}{S}])^{C_{p^{n  }}}
\xrightarrow{R_\CF}
                  \THH_\CF(\spec{A}[\oid{G}{S}])^{C_{p^{n-1}}}
\]
which is natural in~$S$.
Each of the claimed $\pi_*$-isomorphisms then follows inductively from this, \autoref{split-THH}, and the Five Lemma.
\end{proof}

\begin{proof}[Proof of \autoref{pro-iso-TC}]
From \autoref{split-THH-fix} we deduce the analogous statement for $\TC^n$ as follows.
Using the natural $\pi_*$-isomorphism $\Sigma\hoeq\TO\hocoeq$ and the fact that $\Sigma$ commutes with smash products over the orbit category, it is enough to consider $\hocoeq(R,F)$.
Since the maps $R$ and~$F$ commute with~$\pr_\CF$, and homotopy colimits preserve $\pi_*$-isomorphisms and commute with smash products over the orbit category, the statement for~$\hocoeq(R,F)$ is an immediate corollary of \autoref{split-THH-fix}.
Finally, notice that the family~$\Cyc$ is $p$-radicable.
\end{proof}


\section{Isomorphism and injectivity results}

This section is devoted to the deduction from \autoref{split-THH-fix} of Theorems~\ref{iso-TC-finite-groups}, \ref{inj-TC-technical}, and~\ref{inj-TC-rationally}.
In order to use \autoref{split-THH-fix} to obtain statements about~$\TC$, we need to study the following commutative diagram in~$\CW\CT$.
Assume that the family~$\CF$ is $p$-radicable.
\begin{equation}
\label{eq:holim-smash}
\begin{tikzcd}[column sep=4em]
\ds EG(\CF)_+\sma_{\Or G}\TC(\spec{A}[\oid{G}{-}];p)
\arrow[r, "\asbl"]
\arrow[d, "t"']
&
\TC     (\spec{A}[G];p)
\arrow[d, equal]
\\
\ds\holim_\RFcat\Bigl(
EG(\CF)_+\sma_{\Or G}\bigl(\THH    (\spec{A}[\oid{G}{-}])\bigr)^{C_{p^n}}
\Bigr)
\arrow[r, "\holim(\asbl)", "\ts\three"']
\arrow[d, "\holim(\id\sma\pr_\CF)"', "\ts\one"]
&
\ds\holim_\RFcat
\THH    (\spec{A}[G];p)^{C_{p^n}}
\arrow[d, "\pr_\CF"]
\\
\ds\holim_\RFcat\Bigl(
EG(\CF)_+\sma_{\Or G}\bigl(\THH_\CF(\spec{A}[\oid{G}{-}])\bigr)^{C_{p^n}}
\Bigr)
\arrow[r, "\holim(\asbl)"', "\ts\two"]
&
\ds\holim_\RFcat
\THH_\CF(\spec{A}[G];p)^{C_{p^n}}
\end{tikzcd}
\hspace{-.2ex}
\end{equation}
The bottom square is obtained by taking the homotopy limit of the diagrams from \autoref{split-THH-fix}.
Therefore the maps $\one$ and~$\two$ are $\pi_*$-isomorphisms, and so $\three$ is split injective.
Moreover, if $\Cyc\subseteq\CF$ then $\pr_\CF$ is an isomorphism and so $\three$ is a $\pi_*$-isomorphism.
It remains to analyze the map~$t$, the natural map that interchanges the order of smashing over~$\Or G$ and taking $\holim_\RFcat$.

\begin{theorem}
\label{holim-smash}
Consider the map~$t$ in diagram~\eqref{eq:holim-smash}.
\begin{enumerate}
\item\label{i:holim-smash-fin-type}
If $EG(\CF)$ is of finite type then~$t$ is a $\pi_*$-isomorphism.
\item\label{i:holim-smash-afg}
Let $N\ge0$ be an integer.
Assume that conditions~\ref{i:F-in-Fin}, \ref{i:(F)-finite}, and~\ref{techA} in \autoref{inj-TC-rationally} are satisfied.
Then~$t$ is a $\pi_n^\IQ$-isomorphism for all~$n\le N$.
\end{enumerate}
\end{theorem}

\begin{proof}
This question is studied in detail in~\cite{LRV}.

\ref{i:holim-smash-fin-type} is unfortunately not stated explicitly in~\cite{LRV}, but it follows from
the inductive argument given in~\cite{LRV}*{Section~5}.
The only difference is that the $G$-CW~complex~$X$ there is now assumed to be of finite type, and therefore the proof of~\cite{LRV}*{Lemma~5.1} simplifies drastically, as we proceed to explain.
The proof remains unchanged until the discussion of the Atiyah-Hirzebruch spectral sequence.
But then one shows directly that the maps
\begin{equation}
\label{eq:sseq}
H_p^{\IZ\Or G} \Bigl( X ; \smallprod_{i \in I} \pi_q ( \BE (c_i) ) \Bigr)
\TO
\smallprod_{i \in I}  H_p^{\IZ\Or G} \bigl( X ;  \pi_q ( \BE (c_i) ) \bigr)
\end{equation}
induced from the projections $\prod_{i \in I} \pi_q ( \BE (c_i) ) \TO \pi_q ( \BE (c_i) )$ are all isomorphisms.
This follows because now the cellular $\IZ\Or G$-chain complex $C_{\ast}^{\IZ\Or G} ( X)$ is in each degree a finitely generated free $\IZ\Or G$-module, and hence the natural map
\[
C_{\ast}^{\IZ\Or G} ( X) \tensor_{\IZ\Or G} \smallprod_{i \in I} \pi_q ( \BE (c_i) )
\TO
\smallprod_{i \in I} C_{\ast}^{\IZ\Or G} ( X) \tensor_{\IZ\Or G} \pi_q ( \BE (c_i) )
\]
is an isomorphism.
Since $\prod_{i \in I}$ preserves exact sequences, the isomorphism~\eqref{eq:sseq} follows, completing the proof of~\ref{i:holim-smash-fin-type}.

For~\ref{i:holim-smash-afg}, we now verify the assumptions (A) to~(D) of~\cite{LRV}*{Addendum~1.3, page~140}, which imply that~$\pi_n(t)$ is an almost isomorphism, and so in particular a rational isomorphism, for all~$n\leq N$.
For assumption~(A), the category~$\CC=\RFcat$ has a $2$-dimensional model for~$E\RFcat$ by~\cite{LRV}*{Proposition~7.3, page~162}.
For~(B), the fundamental fibration sequence (see e.g.~\cite{kc}*{Theorem~8.1(i), page~973}) implies inductively that $(\THH(\spec{A}[\oid{G}{-}]))^{C_{p^n}}$ is always $(-1)$-connected.
For $X=EG(\CF)$, (C) is implied by our assumptions on~$\CF$.
And finally (D) follows from our assumptions on~$H_s(BZ_GH;\IZ)$ combined with~\cite{LRV}*{Proposition~1.7, page~142}.
\end{proof}

Combining \autoref{split-THH-fix}, diagram~\eqref{eq:holim-smash}, and \autoref{holim-smash}\ref{i:holim-smash-fin-type} we immediately obtain the following corollary.
Notice that \autoref{inj-TC-technical}\ref{i:split-inj-TC-technical} is a special case of this.

\begin{corollary}
\label{split-TC-radicable-fin-type}
If~$\CF$ is $p$-radicable and there is a universal space~$EG(\CF)$ of finite type,
then the assembly map
\[
EG(\CF)_+\sma_{\Or G}\TC(\spec{A}[\oid{G}{-}];p)
\TO
\TC(\spec{A}[G];p)
\]
is split injective.
If $\CF$ contains all cyclic groups, then it is a $\pi_*$-isomorphism.
\end{corollary}

We are now ready to prove Theorems~\ref{iso-TC-finite-groups}, \ref{inj-TC-technical}\ref{i:pi*-inj-TC-technical}, and~\ref{inj-TC-rationally}.

\begin{proof}[Proof of \autoref{iso-TC-finite-groups}]
This follows at once from \autoref{split-TC-radicable-fin-type} and the following claim:
If $G$ is finite, then there is a universal space~$EG(\CF)$ of finite type for any family~$\CF$.

In order to prove this claim, consider the category~$\CC(G,\CF)$ of pointed orbits $(G/H,gH)$ with~$H\in\CF$ and pointed $G$-maps $(G/H,gH)\TO(G/H',g'H')$.
The group~$G$ acts on~$\CC(G,\CF)$ by left multiplication.

It is well known that $\real{N_\bullet\CC(G,\CF)}$ is a model for~$EG(\CF)$; see for example~\cite{Ramras}*{Section~2} and the other references given there in Section~1. 
This proves the claim because $\CC(G,\CF)$ is finite when~$G$ is finite.

We briefly recall the argument that $\real{N_\bullet\CC(G,\CF)}$ satisfies the characterizing property of~$EG(\CF)$.
For a subgroup $K\le G$ we have $\real{N_\bullet\CC(G,\CF)}^K = \real{N_\bullet(\CC(G,\CF)^K)}$.
It is straightforward to check that $\CC(G,\CF)^K$ is the full subcategory on objects $(G/H,gH)$ such that $g^{-1}Kg\le H$.
Hence the subcategory is empty if $K\notin\CF$ and it has the initial object $(G/K,eK)$ if $K\in\CF$.
\end{proof}

\begin{proof}[Proof of \autoref{inj-TC-technical}\ref{i:pi*-inj-TC-technical}]
Assume that $\CF$ is the directed union~$\bigcup_{j\in\CJ}\CF_j$ and that, for each~$j\in\CJ$, there is a model for~$EG(\CF_j)$ of finite type and $\CF_j$ is $p$-radicable.
Using functorial models for~$EG(\CF_j)$ (e.g., those described in the previous proof, which are usually not of finite type), we obtain a functor from $\CJ$ to~$G$-spaces whose homotopy colimit
\begin{equation}
\label{eq:EGF-as-hocolim}
\hocolim_{j\in\CJ}EG(\CF_j)
\end{equation}
is a model for~$EG(\CF)$.
This leads to the following commutative diagram.
\[
\begin{tikzcd}[column sep=-2em, row sep=large]
&
\ds\mathllap{\hocolim_{j\in\CJ}}\Bigl(EG(\CF_j)_+\sma_{\Or G}\TC(\spec{A}[\oid{G}{-}];p)\Bigr)
\arrow[rd, "\hocolim(\asbl_{\CF_j})" pos=.55]
\arrow[ld, "s"']
\\
\ds\Bigl(\hocolim_{j\in\CJ}EG(\CF_j)\Bigr)_+\sma_{\Or G}\TC(\spec{A}[\oid{G}{-}];p)
\arrow[rr, "\asbl_\CF"']
&&
\ds\TC(\spec{A}[G];p)
\end{tikzcd}
\]
Here~$s$ is the natural $\pi_*$-isomorphism that interchanges the order of smashing over~$\Or G$ and taking homotopy colimits.
Since~\eqref{eq:EGF-as-hocolim} is a model for~$EG(\CF)$, the horizontal map is a model for the assembly map with respect to~$\CF$.
For each~$j\in\CJ$, \autoref{split-TC-radicable-fin-type} applies to~$\CF_j$ by assumption, and so $\asbl_{\CF_j}$ is split injective, regardless of what model for~$EG(\CF_j)$ is used.
Since homotopy groups commute with directed (homotopy) colimits, and directed colimits are exact, we conclude that $\hocolim(\asbl_{\CF_j})$ is $\pi_*$-injective, and so the same is true for~$\asbl_\CF$, completing the proof.
\end{proof}

\begin{proof}[Proof of \autoref{inj-TC-rationally}]
This follows at once from \autoref{split-THH-fix}, diagram~\eqref{eq:holim-smash}, and \autoref{holim-smash}\ref{i:holim-smash-afg}.
\end{proof}

We conclude this section with an application and a remark.
Invoking the Transitivity Principle for assembly maps~\cite{LR-survey}*{Theorem~65, page~742}, \autoref{iso-TC-finite-groups} directly implies the following result.

\begin{corollary}
For any group~$G$ the relative assembly map from finite cyclic subgroups to all finite subgroups
\[
EG(\FCyc)_+\sma_{\Or G}\TC(\spec{A}[\oid{G}{-}];p)
\TO
EG(\Fin )_+\sma_{\Or G}\TC(\spec{A}[\oid{G}{-}];p)
\]
is a $\pi_*$-isomorphism.
\end{corollary}

The same strategy does not apply to prove the analogous result for the map
\[
EG(\Cyc )_+\sma_{\Or G}\TC(\spec{A}[\oid{G}{-}];p)
\TO
EG(\VCyc)_+\sma_{\Or G}\TC(\spec{A}[\oid{G}{-}];p)
\,.
\]
In fact, \autoref{CpxCoo} shows that there are virtually cyclic groups without a universal space~$EG(\Cyc)$ of finite type.
Therefore we cannot apply \autoref{split-TC-radicable-fin-type} to conclude that the assembly map with respect to the family of cyclic subgroups is a $\pi_*$-isomorphism for all virtually cyclic groups.

\begin{example}
\label{CpxCoo}
Consider the group $G=\IZ/p\IZ\oplus\IZ$.
We claim that there are infinitely many distinct maximal infinite cyclic subgroups~$C\le G$.
This claim implies that $EG(\Cyc)$ cannot be of finite type, because for any maximal infinite cyclic subgroup~$C\le G$, since $EG(\Cyc)^C\ne\emptyset$, there must be a zero-cell~$G/H$ such that $(G/H)^C\ne\emptyset$, i.e., $H=C$.
Therefore $EG(\Cyc)$ must have infinitely many $G$-zero-cells.

To prove the claim, write $C(i,m)$ for the subgroup of~$G$ generated by~$(i,m)$.
Every infinite cyclic subgroup is of the form~$C(i,m)$ for precisely one pair~$(i,m)$ with $i\in\IZ/p\IZ$ and~$m\ge1$.
Clearly, $C(i,m)\leq C(j,n)$ if and only if there exists a $k\ge1$ such that $(i,m)=(kj,kn)$.
If $(i,p^r)=(kj,kn)$, then either $n=p^r$ and hence $i=j$, or $n=p^s$ with $s<r$ and $i=kj=p^{r-s}j=0\in\IZ/p\IZ$.
Therefore the
\begin{equation}
\label{eq:CpxCoo}
C(i,p^r)
\quad
\text{with $i\in\IZ/p\IZ-\{0\}$ and~$r\ge0$}
\end{equation}
are infinitely many distinct maximal infinite cyclic subgroups.
Notice also that, if $m=kp^r$ with $(k,p)=1$, then $C(i,m)\leq C(ik^{-1},p^r)$. Therefore the maximal infinite cyclic subgroups are precisely~$C(0,1)$ and the subgroups listed in~\eqref{eq:CpxCoo}.
\end{example}

An affirmative answer to the following question would imply that only for one infinite group can \autoref{split-TC-radicable-fin-type} be used to prove nontrivial isomorphism results.

\begin{question}
\label{EG(Cyc)}
Is it true that there exists a universal space $EG(\Cyc)$ of finite type if and only if~$G$ is either finite or infinite cyclic or infinite dihedral?
\end{question}

This question is similar to a conjecture about~$EG(\VCyc)$ posed by Juan-Pineda and Leary~\cite{JPL}*{Conjecture~1, page~142}.
In our original formulation of this question the infinite dihedral case was erroneously missing.
The correct formulation of \autoref{EG(Cyc)} given above is due to Puttkamer and Wu~\cite{vPW}*{Question~A}, who construct an $EG(\Cyc)$ of finite type for the infinite dihedral group in~[loc.~cit., Lemma~3.9].
Moreover, in~[loc.~cit., Theorem~II] they answer \autoref{EG(Cyc)} affirmatively for all of the following groups:
elementary amenable groups;
one-relator groups;
3-manifold groups;
acylindrically hyperbolic groups;
CAT(0) cube groups;
linear groups.


\section{From finite to virtually cyclic subgroups}
\label{EG(VCyc)}

The purpose of this section is to prove the following result, which shows that \autoref{inj-TC-technical}\ref{i:pi*-inj-TC-technical} implies \autoref{inj-TC-examples}\ref{i:VCyc}.

\begin{proposition}
\label{hyperbolic-or-vfga-satisfy-technical}
The family~$\CF=\VCyc$ satisfies the assumption of \autoref{inj-TC-technical}\ref{i:pi*-inj-TC-technical} in each of the following two cases:
\begin{enumerate}
\item
\label{i:hyperbolic}
$G$ is hyperbolic;
\item
\label{i:vfga}
$G$ is virtually finitely generated abelian.
\end{enumerate}
\end{proposition}

In fact, we are going to prove~\ref{i:vfga} directly, and deduce~\ref{i:hyperbolic} from a general criterion formulated in \autoref{NM-satisfies-technical}.
In order to formulate this criterion we need to introduce some notation.
Given any subgroup~$H\le G$, we denote by~$N_GH$ the normalizer of~$H$ in~$G$, and we define its Weyl group as the quotient $W_GH=N_GH/H$.
(Warning: this definition agrees with the one used in~\citelist{\cite{L-type} \cite{L-survey} \cite{LWeiermann} \cite{LRosenthal}} but not with~\cite{kc}, where the Weyl group is taken to be the quotient~$N_GH/(Z_GH\cdot H)$.)
We also use the following definitions from~\cite{LWeiermann}*{Notation 2.7, page~504}.

\begin{definition}
\label{NM}
Let $\CF\subset\CV$ be families of subgroups of a group~$G$.
We consider the following two conditions:
\begin{itemize}[leftmargin=\widthof{$\MNM{N}{\CF}{\CH}\quad$}]
\item[$\MNM{}{\CF}{\CV}:$]
each~$V\in\CV-\CF$ is contained in a unique maximal $V^{\max}\in\CV-\CF$;
\item[$\MNM{N}{\CF}{\CV}:$]
each~$V\in\CV-\CF$ is contained in a unique maximal $V^{\max}\in\CV-\CF$
and\newline we have $N_G(V^{\max})=V^{\max}$.
\end{itemize}
\end{definition}

\begin{lemma}
\label{hyperbolic}
Let $G$ be a hyperbolic group. Then:
\begin{enumerate}
\item
\label{i:hyperbolic-implies-NM}
$G$ satisfies condition $\MNM{N}{\Fin}{\VCyc}$;
\item
\label{i:hyperbolic-maxinfcyc}
if $G$ is not virtually cyclic, then $G$ contains infinitely many conjugacy classes of maximal infinite cyclic subgroups;
\item
\label{i:hyperbolic-maxvircyc}
if $G$ is not virtually cyclic, then $G$ contains infinitely many conjugacy classes of maximal infinite virtually cyclic subgroups.
\end{enumerate}
\end{lemma}

\begin{proof}
\ref{i:hyperbolic-implies-NM} is proved in~\cite{LWeiermann}*{Theorem~3.1 and Example~3.6, pages~506--510}, and \ref{i:hyperbolic-maxinfcyc} is proved in~\cite{Gromov}*{Corollary~5.1.B, pages~136--137, or Corollary~8.2.G, page~213}; notice that the elementary hyperbolic groups are precisely those that are virtually cyclic.
Together these imply~\ref{i:hyperbolic-maxvircyc} as follows.
Suppose by contradiction that $\{C_j\}_{j\in\IN}$ is an infinite collection of maximal infinite cyclic subgroups which are pairwise not conjugate but such that all $C_j^{\max}$ are conjugate, i.e., for each~$j\in\IN$ there is a~$g_j\in G$ such that $g_jC_j^{\max}g_j^{-1}=C_0^{\max}$.
Then $\{g_jC_jg_j^{-1}\}_{j\in\IN}$ is an infinite collection of pairwise distinct maximal infinite cyclic subgroups of the virtually cyclic group~$C_0^{\max}$.
But any virtually cyclic group contains only finitely many maximal cyclic subgroups.
\end{proof}

\begin{example}
Suppose that every subgroup of~$G$ which is not virtually cyclic contains a non-abelian free subgroup.
Then $G$ satisfies $\MNM{N}{\Fin}{\VCyc}$ by~\cite{LWeiermann}*{Theorem~3.1 and Lemma~3.4, pages~506--509}.
\end{example}

\begin{lemma}
\label{NM-satisfies-technical}
Assume that $G$ satisfies condition $\MNM{}{\Fin}{\VCyc}$.
Assume that there are models of finite type for~$EG(\Fin)$ and for~$EW_G V$ for each maximal infinite virtually cyclic subgroup~$V$ of~$G$.
Then the family~$\VCyc$ satisfies the assumption of \autoref{inj-TC-technical}\ref{i:pi*-inj-TC-technical}.
\end{lemma}

\begin{proof}
Choose a complete set of representatives~$\CM$ of the conjugacy classes of maximal infinite virtually cyclic subgroups of~$G$.
Condition $\MNM{}{\Fin}{\VCyc}$ says that for any infinite virtually cyclic subgroup $H\le G$ there is exactly one $V\in\CM$ such that $gHg^{-1}\le V$ for some~$g\in G$, or, equivalently, such that $(G\times_{N_GV}p_V^*EW_GV)^H\ne\emptyset$.
Here $\MOR{p_V}{N_GV}{W_GV}$ denotes the projection and $\MOR{p_V^*}{\Top^{W_GV}}{\Top^{N_GV}}$ the corresponding restriction functor.
By~\cite{LWeiermann}*{Corollary~2.10, page~505}, there is a $G$-pushout
\begin{equation}
\label{eq:EG(VCyc)-hyperbolic}
\begin{tikzcd}
\ds\smash{\smallcoprod_{V\in\CM}}G\timesd_{N_GV}EN_GV(\Fin)
\arrow[d, "\pr"']
\arrow[r]
&
EG(\Fin)
\arrow[d]
\\
\ds\smallcoprod_{V\in\CM}G\timesd_{N_GV}p_V^*EW_GV
\arrow[r]
&
EG(\VCyc)
\mathrlap{\,.}
\end{tikzcd}
\end{equation}

Now let $\CP_f(\CM)$ be the directed poset of finite subsets of $\CM$ ordered by inclusion.
Given $\CS\in\CP_f(\CM)$, let $\CF_\CS$ be the family of subgroups of~$G$ that are either finite or subconjugate to some~$V\in\CS$.
Clearly
\[
\VCyc=\bigcup_{\CS\in\CP_f(\CM)}\CF_\CS
\]
and each family~$\CF_\CS$ is $p$-radicable.
So it only remains to construct universal spaces $EG(\CF_\CS)$ of finite type.

As in~\eqref{eq:EG(VCyc)-hyperbolic} we obtain a $G$-pushout
\[
\begin{tikzcd}
\ds\smash{\smallcoprod_{V\in\CS}}G\timesd_{N_GV}EN_GV(\Fin)
\arrow[d, "\pr"']
\arrow[r]
&
EG(\Fin)
\arrow[d]
\\
\ds\smallcoprod_{V\in\CS}G\timesd_{N_GV}p_V^*EW_GV
\arrow[r]
&
EG(\CF_\CS)
\mathrlap{\,.}
\end{tikzcd}
\]
We are assuming that there are models of finite type for $EG(\Fin)$ and each~$EW_GV$.
Since $V$ is maximal virtually cyclic, $W_GV$ is torsionfree, and hence there is a model of finite type for~$EN_GV(\Fin)$ by~\cite{L-type}*{Theorem~3.2, page~193}.
Since $\CS$ is finite, we conclude that $EG(\CF_\CS)$ is of finite type.
\end{proof}

We can now finish the proof of \autoref{hyperbolic-or-vfga-satisfy-technical}.

\begin{proof}[Proof of \autoref{hyperbolic-or-vfga-satisfy-technical}]
\ref{i:hyperbolic} follows from \autoref{NM-satisfies-technical}, \autoref{hyperbolic}\ref{i:hyperbolic-implies-NM}, and the fact that hyperbolic groups have universal spaces~$EG(\Fin)$ of finite type by~\cite{Meintrup-Schick}.

\ref{i:vfga}
We follow the arguments in~\cite{LRosenthal}*{Subsections 1.4 and~3.3, pages~1571 and~1584--1585}.
By assumption there is a group extension
\begin{equation}
\label{eq:A->G->Q}
1 \TO A \TO G \TO[q] Q \TO 1
\,,
\end{equation}
where $A$ is finitely generated free abelian and $Q$ is finite.
The conjugation action of $G$ on the normal abelian subgroup $A$ induces
an action $\MOR{\rho}{Q}{\aut(A)}$.

Let $\CM$ be the set of maximal infinite cyclic subgroups of $A$.
Since any automorphism of $A$ sends a maximal infinite cyclic subgroup to a maximal infinite cyclic subgroup,
$\rho$ induces a $Q$-action on~$\CM$.
Fix a subset
\(
\CN\subseteq\CM
\)
whose intersection with each $Q$-orbit in~$\CM$ consists of precisely one element.

Given~$C\in\CN$, denote by
\(
Q_C\le Q
\)
the isotropy group of~$C$ under the $Q$-action.
The given extension~\eqref{eq:A->G->Q} induces an extension
\[
1 \TO A/C \TO W_GC \TO Q_C \TO 1
\,,
\]
Since $C\le A$ is a maximal infinite cyclic subgroup, $A/C$ is finitely generated free abelian.

Notice that any infinite cyclic subgroup $C\le A$ is contained in a unique maximal infinite cyclic subgroup $C^{\max}\le A$.
In particular, for two maximal infinite cyclic subgroups $C,D\le A$, either $C \cap D = \{0\}$ or $C = D$.
For every $C\in\CN$ we have
\[
N_GC = \SET{g\in G}{\#(gCg^{-1}\cap C)=\infty} = q^{-1}(Q_C).
\]

Now we define an equivalence relation on the set of infinite virtually cyclic subgroups of $G$.
We say that $V_1$ and~$V_2$ are equivalent if and only if $(A\cap V_1)^{\max} = (A\cap V_2)^{\max}$.
Then for every infinite virtually cyclic subgroup $V\le G$ there is exactly one $C\in\CN$ such that $V$ is equivalent to~$gCg^{-1}$ for some $g\in G$.
We obtain from~\cite{LWeiermann}*{Theorem~2.3, page~502} a $G$-pushout
\begin{equation}
\label{eq:EG(VCyc)-vfga}
\begin{tikzcd}
\ds\smash{\smallcoprod_{C\in\CN}}G\timesd_{N_GC}EN_GC(\Fin)
\arrow[d, "\pr"']
\arrow[r]
&
EG(\Fin)
\arrow[d]
\\
\ds\smallcoprod_{C\in\CN}G\timesd_{N_GC}p_C^*EW_GC(\Fin)
\arrow[r]
&
EG(\VCyc)
\mathrlap{\,,}
\end{tikzcd}
\end{equation}
where $\MOR{p_C}{N_GC}{W_GC}$ is the projection.

We now proceed as in the proof of \autoref{NM-satisfies-technical}.
Let $\CP_f(\CN)$ be the directed poset of finite subsets of $\CN$ ordered by inclusion.
Given $\CS\in\CP_f(\CN)$, let $\CF_\CS$ be the family of subgroups that are either finite, or infinite virtually cyclic and equivalent to a conjugate of some~$C\in\CS$.
Clearly
\[
\VCyc=\bigcup_{\CS\in\CP_f(\CN)}\CF_\CS
\]
and each family~$\CF_\CS$ is $p$-radicable.

For each $\CS\in\CP_f(\CN)$ we obtain from~\eqref{eq:EG(VCyc)-vfga} a $G$-pushout
\[
\begin{tikzcd}
\ds\smash{\smallcoprod_{C\in\CS}}G\timesd_{N_GC}EN_GC(\Fin)
\arrow[d, "\pr"']
\arrow[r]
&
EG(\Fin)
\arrow[d]
\\
\ds\smallcoprod_{C\in\CS}G\timesd_{N_GC}p_C^*EW_GC(\Fin)
\arrow[r]
&
EG(\CF_\CS)
\mathrlap{\,.}
\end{tikzcd}
\]
Any virtually finitely generated abelian group~$\Gamma$ admits a surjection with finite kernel onto a crystallographic group (see e.g.~\cite{Quinn-abelian}*{Lemma~4.2.1, page~182}), and therefore it has a universal space~$E\Gamma(\Fin)$ of finite type.
So we can chose models of finite type for $EN_GC(\Fin)$, $EW_GC(\Fin)$, and $EG(\Fin)$, and obtain a finite type $EG(\CF_\CS)$ for each $\CS\in\CP_f(\CN)$.
\end{proof}


\section{Failure of surjectivity}

This section is devoted to the proof of \autoref{not-surj-TC} and its variant for~$\tr$.
More precisely, we establish the following result.

\begin{theorem}
\label{not-surj-TR-TC}
The assembly maps for the family of virtually cyclic subgroups
\[
\MOR{\asbl^\TR}{EG(\VCyc)_+\sma_{\Or G}\TR(\spec{A}[\oid{G}{-}];p)}{\TR(\spec{A}[G];p)}
\]
and
\[
\MOR{\asbl^\TC}{EG(\VCyc)_+\sma_{\Or G}\TC(\spec{A}[\oid{G}{-}];p)}{\TC(\spec{A}[G];p)}
\]
are not always $\pi_*$-surjective.
For example, if $\spec{A}=\IZ_{(p)}$ and $G$ is either finitely generated free abelian or torsion-free hyperbolic, but not cyclic, then neither $\pi_0(\asbl^\TR)$ nor $\pi_{-1}(\asbl^\TC)$ is surjective.
\end{theorem}

In fact, the assumptions on $\spec{A}$ and~$G$ can be relaxed, as explained in \autoref{more-counterexamples}.

We begin with the following general observations, relating the analysis of the assembly maps for $\tr$ and~$\tc$ to the assembly maps for the corresponding ``Whitehead'' theories.

Let $\MOR{\T}{\Groupoids}{\IN\Sp}$ be a functor, e.g., $\T=\TC(\spec{A}[-];p)$.
Assume that $\T$ preserves equivalences, i.e., it sends equivalences of groupoids to $\pi_*$-isomorphisms.
Given any space~$X$ there is the classical assembly map
\begin{equation}
\label{eq:classical-asbl-Pi(X)}
X_+\sma\T(1)\TO\T(\Pi(X))
\,,
\end{equation}
where $\Pi(X)$ denotes the fundamental groupoid of~$X$, and $1$ denotes the trivial groupoid.
We define $\Wh^\T(X)$ to be the homotopy cofiber of~\eqref{eq:classical-asbl-Pi(X)}.

Given a group~$G$, consider the composition
\[
\Or G
\xrightarrow{\oid{G}{-}}
\Groupoids
\xrightarrow{\ B\ }
\Top
\,,
\]
where $B=\real{N_\bullet(-)}$ is the classifying space functor,
and obtain the diagram
\begin{equation}
\label{eq:classical-asbl-Wh}
B(\oid{G}{-})_+ \sma \T(1)
\TO
\T(\Pi(B(\oid{G}{-}))
\TO
\Wh^\T(B(\oid{G}{-}))
\end{equation}
of functors $\Or G\TO\IN\Sp$.
By definition, \eqref{eq:classical-asbl-Wh} is objectwise a homotopy cofibration sequence.
To simplify the notation, we let $\Wh^\T(H)=\Wh^\T(B(\oid{G}{\,(G/H)}))$.

\begin{lemma}
\label{Wh-trick}
Assume that the assembly map
\[
\MOR{\pi_q\bigl(\asbl^{\T}\bigr)}{\pi_q\Bigl(EG(\CF)_+\sma_{\Or G}\T\Bigr)}{\pi_q\bigl(\T(G)\bigr)}
\]
is injective for some $q\in\IZ$.
Then:
\begin{enumerate}
\item
\label{i:Wh-trick-inj}
the assembly map
\[
\MOR{\pi_q\bigl(\asbl^{\Wh^{\T}}\bigr)}{\pi_q\Bigl(EG(\CF)_+\sma_{\Or G}\Wh^\T\Bigr)}{\pi_q\bigl(\Wh^\T(G)\bigr)}
\]
is also injective;
\item
\label{i:Wh-trick-surj}
$\pi_q\bigl(\asbl^{\T}\bigr)$ is surjective if and only if $\pi_q\bigl(\asbl^{\Wh^{\T}}\bigr)$ is surjective.
\end{enumerate}
\end{lemma}

\begin{proof}
\ref{i:Wh-trick-inj} is proved in~\cite{kc}*{Proof of Addendum~1.18, pages~1010--1011} in the special case when $\T=\K^{\ge0}(\IZ[-])$, but the argument works without changes for any~$\T$.

\ref{i:Wh-trick-surj} is an immediate consequence of the proof of~\ref{i:Wh-trick-inj} in loc.~cit.
\end{proof}

\begin{lemma}
\label{Wh-holim}
Let \[\dotsb\TO\T^{n+1}\TO\T^{n}\TO\dotsb\TO\T^2\TO\T^1\] be a sequence of functors $\Groupoids\TO\IN\Sp$ that preserve equivalences, and let $\T=\holim_{n}\T^{n}$.
If a group $G$ has a classifying space~$BG$ that is a finite CW complex, then there is a natural $\pi_*$-isomorphism
\[
\Wh^\T(G)\TO\holim_{n\in\IN}\Wh^{\T^n\!}(G)
\,.
\]
\end{lemma}

\begin{proof}
This follows easily from the definitions, noticing that homotopy limits commute up to $\pi_*$-isomorphisms with homotopy cofibers and with the functor $X_+\sma-$, provided that $X$ is a finite CW complex.
\end{proof}

Now we specialize to the case $\T=\TC(\spec{A}[-];p)=\holim_{n}\TC^n(\spec{A}[-];p)$.
We use the abbreviations
$\Wh^\TC=\Wh^{\TC(\spec{A}[-];p)}$
and
$\wh^\tc_q(G)=\pi_q(\Wh^\TC(G))$,
$q\in\IZ$, and similarly for $\TC^n$ and $\TR^n$.

The following result provides the essential computations used in the proof of \autoref{not-surj-TR-TC}.
This result is based on theorems of Hesselholt and Madsen about topological cyclic homology of polynomial rings~\citelist{\cite{HM-mixed}*{Theorem~C, page~4} \cite{H-S1}*{Theorem~2, page~139}}.

\begin{theorem}
\label{keylemma}
Assume that~$\spec{A}$ is a connective ring spectrum whose homotopy groups are $\IZ_{(p)}$-modules.
Let~$C$ be an infinite cyclic group.
\begin{enumerate}
\item
\label{i:ker(R)-Wh-TR}
For each $n\ge2$ there is a short exact sequence
\[
0
\TO
\bigoplus_{\substack{j\in\IZ-p\IZ\\-n<t<\infty}}
\pi_0\spec{A}
\TO
\wh_0^{\tr^{n\!}}(C)
\TO[R]
\wh_0^{\tr^{n-1\!}}(C)
\TO
0
\,.
\]

\item
\label{i:ker(R)-Wh-TC}
Assume that $\pi_0\spec{A}\cong\IZ_{(p)}$.
Then for each~$n\ge3$ there is a short exact sequence
\[
0
\TO
\bigoplus_{j\in\IZ-p\IZ}
\IF_p
\TO
\wh_{-1}^{\tc^{n\!}}(C)
\TO[R]
\wh_{-1}^{\tc^{n-1\!}}(C)
\TO
0
\,.
\]
\end{enumerate}
\end{theorem}

\begin{proof}
Define
\begin{equation}
\label{eq:Mnq}
M^n_q=\bigoplus_{j\in\IZ-p\IZ}\bigl(\tr^n_{q}(\spec{A};p)\oplus\tr^n_{q-1}(\spec{A};p)\bigr)
.
\end{equation}
With this notation \cite{H-S1}*{Theorem~2, page~139} describes, in particular, an isomorphism of abelian groups
\begin{equation}
\label{eq:TRn-Laurent}
\tr^n_{q}(\spec{A};p)\oplus\tr^n_{q-1}(\spec{A};p)
\oplus
\bigoplus_{-n<s<0}M^{n+s}_q
\oplus
\bigoplus_{0\le t<\infty}M^{n}_q
\TO[\cong]
\tr^n_q(\spec{A}[C];p)
\end{equation}
for any symmetric ring spectrum~$\spec{A}$ whose homotopy groups are $\IZ_{(p)}$-modules.
To see that the source of the isomorphism~\eqref{eq:TRn-Laurent} agrees with the group in~\cite{H-S1}*{first display after Theorem~2, page~139}, one just needs to rewrite the index set~$\IZ$ for the first direct sum in~loc.~cit.\ as the disjoint union
\[
\{0\}\cup\bigcup_{t\ge0}\SET{jp^t}{j\in\IZ-p\IZ}
\]
and replace $s$ with $-s$.
For a chosen generator~$x$ of~$C$, the map~\eqref{eq:TRn-Laurent} sends
\begin{gather*}
\Bigl(
a,b,
\bigl(
(a_{s,j},b_{s,j})_{j\in\IZ-p\IZ}
\bigr)_{-n<s<0}
\,,
\bigl(
(a_{t,j},b_{t,j})_{j\in\IZ-p\IZ}
\bigr)_{0\le t<\infty}
\Bigr)
\\\shortintertext{to}
a[x]_{n}^{0}
+
b\,d\!\log[x]_n
+
\adjustlimits\sum_{-n<s<0}\sum_{j\in\IZ-p\IZ}
\Bigl(
V^{-s}
\bigl(
a_{s,j}[x]_{n+s}^{j}
\bigr)
+
dV^{-s}
\bigl(
b_{s,j}[x]_{n+s}^{j}
\bigr)
\Bigr)
\\
+
\adjustlimits\sum_{0\le t<\infty}\sum_{j\in\IZ-p\IZ}
\Bigl(
a_{t,j}[x]_{n}^{jp^t}
+
b_{t,j}[x]_{n}^{jp^t}d\!\log[x]_n
\Bigr)
.
\end{gather*}

As remarked in loc.~cit., the image of the classical assembly map
\[
\tr^n_{q}(\spec{A};p)\oplus\tr^n_{q-1}(\spec{A};p)
\cong
\pi_q\bigl(BC_+\sma\TR^n(\spec{A};p)\bigr)
\TO
\tr^n_q(\spec{A}[C];p)
\]
corresponds exactly to the first two direct summands in the source of~\eqref{eq:TRn-Laurent}.
Therefore we get an isomorphism
\begin{equation}
\label{eq:Wh_q-TRn-Laurent}
\bigoplus_{-n<s<0}M^{n+s}_q
\oplus
\bigoplus_{0\le t<\infty}M^{n}_q
\TO[\cong]
\wh^{\tr^{n\!}}_q(C)
\,.
\end{equation}
Explicit formulas for the maps $R$ and~$F$ with respect to the isomorphism~\eqref{eq:TRn-Laurent} are given in \cite{H-S1}*{page~140}.
In particular, both maps respect the decomposition of~$\tr^n_q(\spec{A}[C];p)$ as the direct sum of the image of the classical assembly map and~$\wh^{\tr^{n\!}}_q(C)$.
On~$\wh^{\tr^{n\!}}_q(C)$, with respect to the isomorphism~\eqref{eq:Wh_q-TRn-Laurent}, the maps~$R$ and $F$ are described as follows.
\[
\begin{tikzcd}[column sep=tiny, row sep=.6em]
M^{1}_q\;
\arrow[draw=none, r, "\ts\oplus" description]
&
\;M^{2}_q\;
\arrow[draw=none, r, "\ts\oplus" description]
\arrow[ddd, "R"']
&
\dotsb
\arrow[draw=none, r, "\ts\oplus" description]
&
M^{n-2}_q
\arrow[draw=none, r, "\ts\oplus" description]
\arrow[ddd, "R"']
&
M^{n-1}_q
\arrow[draw=none, r, "\ts\oplus" description]
\arrow[ddd, "R"']
&
\mathmakebox[\widthof{$M^{n-1}_q$}][c]{M^{n}_q}
\arrow[draw=none, r, "\ts\oplus" description]
\arrow[ddd, "R"']
&
\mathmakebox[\widthof{$M^{n-1}_q$}][c]{M^{n}_q}
\arrow[draw=none, r, "\ts\oplus" description]
\arrow[ddd, "R"']
&
\dotsb
\arrow[rrr, "\cong" pos=.42]
&&&
\wh^{\tr^{n\!}}_q(C)
\arrow[ddd, "R"]
\\
\\
\\
&
\;M^{1}_q\;
\arrow[draw=none, r, "\ts\oplus" description]
&
\dotsb
\arrow[draw=none, r, "\ts\oplus" description]
&
M^{n-3}_q
\arrow[draw=none, r, "\ts\oplus" description]
&
M^{n-2}_q
\arrow[draw=none, r, "\ts\oplus" description]
&
M^{n-1}_q
\arrow[draw=none, r, "\ts\oplus" description]
&
M^{n-1}_q
\arrow[draw=none, r, "\ts\oplus" description]
&
\dotsb
\arrow[rrr, "\cong"]
&&&
\wh^{\tr^{n-1\!}}_q(C)
\\
M^{1}_q\;
\arrow[draw=none, r, "\ts\oplus" description]
\arrow[dddr, pos=.2, "L_{-n+1}\mspace{-14mu}"', end anchor={[xshift=.67ex]}]
&
\;M^{2}_q\;
\arrow[draw=none, r, "\ts\oplus" description]
\arrow[dddr, pos=.2, "L_{-n+2}\mspace{-14mu}"', end anchor={[xshift=1.0ex]}]
&
\dotsb
\arrow[draw=none, r, "\ts\oplus" description]
&
M^{n-2}_q
\arrow[draw=none, r, "\ts\oplus" description]
\arrow[dddr, pos=.2, "L_{-2}\mspace{-14mu}"']
&
M^{n-1}_q
\arrow[draw=none, r, "\ts\oplus" description]
\arrow[dddr, pos=.2, "L_{-1}\mspace{-14mu}"']
&
\mathmakebox[\widthof{$M^{n-1}_q$}][c]{M^{n}_q}
\arrow[draw=none, r, "\ts\oplus" description]
\arrow[dddr, pos=.2, "F\!"']
&
\mathmakebox[\widthof{$M^{n-1}_q$}][c]{M^{n}_q}
\arrow[draw=none, r, "\ts\oplus" description]
\arrow[dddr, pos=.2, "F\!"', end anchor={[xshift=.67ex]}]
&
\dotsb
\arrow[rrr, "\cong" pos=.42]
&&&
\wh^{\tr^{n\!}}_q(C)
\arrow[ddd, "F"]
\\
\\
\\
&
\;M^{1}_q\;
\arrow[draw=none, r, "\ts\oplus" description]
&
\mathmakebox[\widthof{$\dotsb$}][c]{\vphantom{M^{1}_q}\dotsb}
\arrow[draw=none, r, "\ts\oplus" description]
&
M^{n-3}_q
\arrow[draw=none, r, "\ts\oplus" description]
&
M^{n-2}_q
\arrow[draw=none, r, "\ts\oplus" description]
&
M^{n-1}_q
\arrow[draw=none, r, "\ts\oplus" description]
&
M^{n-1}_q
\arrow[draw=none, r, "\ts\oplus" description]
&
\mathmakebox[\widthof{$\dotsb$}][c]{\vphantom{M^{1}_q}\dotsb}
\arrow[rrr, "\cong"]
&&&
\wh^{\tr^{n-1\!}}_q(C)
\\
\scriptstyle{\mathclap{s=-n+1}}
&
&
\dotsb
&
\scriptstyle{\mathclap{s=-2}}
&
\scriptstyle{\mathclap{s=-1}}
&
\scriptstyle{\mathclap{t=0}}
&
\scriptstyle{\mathclap{t=1}}
&
\dotsb
\end{tikzcd}
\hspace{-1em}
\]
Here the maps $\MOR{R,F}{M^{n}_q}{M^{n-1}_q}$ respect the direct sum decomposition~\eqref{eq:Mnq} and are given by $R$ and $F$, respectively, on each summand.
The endomorphisms $L_{-n+1},L_{-n+2},\dotsc,L_{-2},L_{-1}$ respect the $\bigoplus_{j\in\IZ-p\IZ}$ decomposition in~\eqref{eq:Mnq}.
The map~$L_{-1}$ is given on the summand indexed by~$j$ by
\begin{align*}
\MOR{\begin{pmatrix}
\ell_p & d+\ell_{(p-1)\eta}
\\
0 & \ell_{(-1)^{q-1}j}
\end{pmatrix}&}%
{\tr^{n-1}_{q}(\spec{A};p)\oplus\tr^{n-1}_{q-1}(\spec{A};p)}%
{\tr^{n-1}_{q}(\spec{A};p)\oplus\tr^{n-1}_{q-1}(\spec{A};p)}
\,;
\\
\shortintertext{the maps~$L_{-n+1},L_{-n+2},\dotsc,L_{-2}$ are given by}
\MOR{\begin{pmatrix}
\ell_p & \ell_{(p-1)\eta}
\\
0 & \id
\end{pmatrix}&}%
{\tr^{n+s}_{q}(\spec{A};p)\oplus\tr^{n+s}_{q-1}(\spec{A};p)}%
{\tr^{n+s}_{q}(\spec{A};p)\oplus\tr^{n+s}_{q-1}(\spec{A};p)}
\end{align*}
on each of the $j$-indexed summands.
Here $d$ is the differential, $\ell_a$ denotes multiplication by~$a$, and $\eta\in\tr^1_1(\spec{S};p)$ is the Hopf class; compare~\cite{H-S1}*{pages~137--138}.
Notice that all maps respect the $\bigoplus_{j\in\IZ-p\IZ}$ decomposition, and with the sole exception of~$L_{-1}$ they do not depend on~$j$.
When $q=0$, the map~$L_{-1}$ does not depend on~$j$ either.

Now we specialize to~$q=0$.
The groups $\tr_0^n(\spec{A};p)$ are isomorphic to~$W_n(\pi_0\spec{A})$, the $p$-typical Witt vectors of length~$n$ of~$\pi_0\spec{A}$, and these isomorphisms commute with the maps $R,F,V$.
So from~\eqref{eq:Wh_q-TRn-Laurent} we obtain isomorphisms
\begin{equation}
\label{eq:Wh_0-TRn-Laurent}
\bigoplus_{j\in\IZ-p\IZ}
\Biggl(
\bigoplus_{-n<s<0}W_{n+s}(\pi_0\spec{A})
\oplus
\bigoplus_{0\le t<\infty}W_{n}(\pi_0\spec{A})
\Biggr)
\TO[\cong]
\wh^{\tr^{n\!}}_0(C)
\,.
\end{equation}
Under~\eqref{eq:Wh_0-TRn-Laurent}, the map $\MOR{R}{\wh^{\tr^{n\!}}_0(C)}{\wh^{\tr^{n-1\!}}_0(C)}$ respects the $\bigoplus_{j\in\IZ-p\IZ}$ decomposition and is given by~$\MOR{R}{W_m(\pi_0\spec{A})}{W_{m-1}(\pi_0\spec{A})}$ on each summand.

Using the short exact sequences
\[
0
\TO
\pi_0\spec{A}
\cong
W_{1}  (\pi_0\spec{A})
\xrightarrow{V^{m-1}}
W_{m}  (\pi_0\spec{A})
\TO[R]
W_{m-1}(\pi_0\spec{A})
\TO
0
\]
we obtain the short exact sequence in~\ref{i:ker(R)-Wh-TR}.

In order to prove~\ref{i:ker(R)-Wh-TC}, consider the following diagram.
\begin{equation}
\label{eq:snake}
\begin{tikzcd}[row sep=.6em, column sep=scriptsize]
\\
0
\arrow[r, dotted]
&
\ds\bigoplus_{\substack{j\in\IZ-p\IZ\\-n<t<\infty}}
\tors_p\pi_0\spec{A}
\arrow[r, dotted]
\arrow[d, shorten <=-.6ex, shorten >=-.6ex]
&
\ker(F-R)
\arrow[r, dotted, "R"]
\arrow[d]
&
\ker(F-R)
\arrow[dddll, dotted, "\partial"' pos=.85, out=0, in=180, looseness=2, overlay]
\arrow[d]
\\
0
\arrow[r]
&
\ds\bigoplus_{\substack{j\in\IZ-p\IZ\\-n<t<\infty}}
\pi_0\spec{A}
\arrow[r]
\arrow[d, shorten <=-.6ex, shorten >=-.6ex, "\lambda"' pos=.36]
&
\wh_{0}^{\tr^{n\!}}(C)
\arrow[r, "R"]
\arrow[d, "F-R"']
&
\wh_{0}^{\tr^{n-1\!}}(C)
\arrow[r]
\arrow[d, "F-R"']
\arrow[l, dashed, bend right, "S"']
&
0
\\
0
\arrow[r]
&
\ds\bigoplus_{\substack{j\in\IZ-p\IZ\\\hphantom{-n}\mathllap{-(n-1)}<t<\infty}}
\pi_0\spec{A}
\arrow[r]
\arrow[d, shorten <=-.6ex, shorten >=-.6ex]
&
\wh_{0}^{\tr^{n-1\!}}(C)
\arrow[r, "R"]
\arrow[d]
&
\wh_{0}^{\tr^{n-2\!}}(C)
\arrow[r]
\arrow[d]
&
0
\\
&
\ds\bigoplus_{\substack{j\in\IZ-p\IZ\\\hphantom{-n}\mathllap{-(n-1)}<t<\infty}}
\pi_0\spec{A}/p\,\pi_0\spec{A}
\arrow[r, dotted]
&
\wh_{-1}^{\tc^{n\!}}(C)
\arrow[r, dotted, "R"]
&
\wh_{-1}^{\tc^{n-1\!}}(C)
\arrow[r, dotted]
&
0
\end{tikzcd}
\end{equation}

The middle two rows are exact by~\ref{i:ker(R)-Wh-TR}.
Let $\lambda$ be the induced map. 
Using the explicit description of the maps $R$ and~$F$ given above and using that $FV=\ell_p$, we see that~$\lambda$ respects the $\bigoplus_{j\in\IZ-p\IZ}$ decomposition and sends the summand indexed by~$t$ to the summand indexed by~$t+1$ via~$\ell_p$.
Notice that $\wh_{-1}^{\tc^{n\!}}(C)=\coker(F-R)$.
Therefore the Snake Lemma produces the dotted exact sequence in~\eqref{eq:snake}.

We are now going to describe explicitly the connecting map~$\partial$.
Fix $j\in\IZ-p\IZ$.
Let
\begin{equation}
\label{eq:a-in-Wh-TR}
a=(a_{-n+2},a_{-n+3},\dotsc,a_{-1},a_0,a_1,\dotsc,a_N,0,0,\dotsc)
\end{equation}
be an element of the summand
\[
W_{1}(\pi_0\spec{A})
\oplus
W_{2}(\pi_0\spec{A})
\oplus
\dotsb
\oplus
W_{n-2}(\pi_0\spec{A})
\oplus
W_{n-1}(\pi_0\spec{A})
\oplus
W_{n-1}(\pi_0\spec{A})
\oplus
\dotsb
\]
indexed by~$j$ in~\eqref{eq:Wh_0-TRn-Laurent}.

Assume that $a\in\ker(F-R)$, i.e., that $a$ is in the upper right corner of~\eqref{eq:snake}.
Using the formulas above, this means that
\begin{equation}
\label{eq:ker(F-R)}
\begin{gathered}
pa_{-n+2}=Ra_{-n+3}
\,,\
\dotsc
\,,\
pa_{-2}=Ra_{-1}
\,,\
pa_{-1}=Ra_{0}
\,,
\\
Fa_{0}=Ra_{1}
\,,\
Fa_{1}=Ra_{2}
\,,\
\dotsc
\,,\
Fa_N=0
\,.
\end{gathered}
\end{equation}

Choose set theoretic sections~$S$ of the surjections~$\MOR{R}{W_m(\pi_0\spec{A})}{W_{m-1}(\pi_0\spec{A})}$ for each~$m$, and define a section~$S$ of $\MOR{R}{\wh^{\tr^{n\!}}_0(C)}{\wh^{\tr^{n-1\!}}_0(C)}$ using them on each summand.
We can assume that $S0=0$.
We obtain that
\begin{multline}
\label{eq:(F-R)(Sx)}
\quad(F-R)Sa=
\\
=(-a_{-n+2},pSa_{-n-2}-a_{-n+3},\dotsc,pSa_{-1}-a_{0},FSa_{0}-a_{1},\dotsc,FSa_N,0,\dotsc)
\,.
\end{multline}
The element $(F-R)Sa$ is in the kernel of $\MOR{R}{\wh^{\tr^{n-1\!}}_0(C)}{\wh^{\tr^{n-2\!}}_0(C)}$ and therefore corresponds to a unique~$b\in\bigoplus_{-(n-1)<t<\infty}\pi_0\spec{A}$ in the summand indexed by the same~$j\in\IZ-p\IZ$.
The class~$[b]$ of~$b$ in the cokernel of~$\lambda$ is by definition the image of~$a$ under the connecting map~$\partial$.

Now assume that $\pi_0\spec{A}\cong\IZ_{(p)}$.
Then $W_n(\IZ_{(p)})$ is a free $\IZ_{(p)}$-module with basis $\SET{V^i(1)}{0\le i\le n-1}$; compare~\cite{HM-mixed}*{Example~1.2.4, page~10}.
The maps $\MOR{F,R}{W_{n}(\IZ_{(p)})}{W_{n-1}(\IZ_{(p)})}$ are $\IZ_{(p)}$-linear and are given by
\begin{alignat}{3}
\label{eq:F-for-Z(p)}
FV^0(1)&=V^0(1)
&\qquad\text{and}\qquad
FV^i(1)&=pV^{i-1}(1)
&\quad
\text{if }&1\le i\le n-1
\,,
\\
\label{eq:R-for-Z(p)}
\qquad
RV^{n-1}(1)&=0
&\qquad\text{and}\qquad
RV^i(1)&=V^i(1)
&
\text{if }&0\le i\le n-2
\,,
\end{alignat}
and of course $V$ is given by~$VV^i(1)=V^{i+1}(1)$.
Even though we do not need it here, we note that the product is determined by
\[
V^i(1)\cdot V^j(1)=p^iV^j(1)
\quad\text{if }0\le i\le j\le n-1
\,.
\]
Define a $\IZ_{(p)}$-linear section~$\MOR{S}{W_{n-1}(\IZ_{(p)})}{W_{n}(\IZ_{(p)})}$ of~$R$ by the obvious formulas~$SV^i(1)=V^i(1)$.
This choice of~$S$ satisfies
\begin{equation}
\label{eq:SF-SR}
S\ell_p=\ell_p S
\,,
\qquad
SF=FS
\,,
\AND
SR-\id=-\prlast
\,,
\end{equation}
where $\prlast$ denotes the projection of~$W_n(\IZ_{(p)})$ onto the subspace generated by $V^{n-1}(1)$, i.e., the kernel of~$R$.
Let~$\MOR{\prlasto}{W_n(\IZ_{(p)})}{\IZ_{(p)}}$ be the map obtained by identifying $\IZ_{(p)}V^{n-1}(1)$ with~$\IZ_{(p)}$.

Using~\eqref{eq:SF-SR} and~\eqref{eq:ker(F-R)}, equation~\eqref{eq:(F-R)(Sx)} now reads
\[
(F-R)Sa
=(-\prlast a_{-n+2},-\prlast a_{-n+3},\dotsc,-\prlast a_{0},-\prlast a_{1},\dotsc)
\,.
\]
(Notice that $\prlast=\id$ on~$W_1(\IZ_{(p)})$.)
Therefore
\begin{equation}
\label{eq:connecting}
\partial a
=(-[\prlasto a_{-n+2}],-[\prlasto a_{-n+3}],\dotsc,-[\prlasto a_{0}],-[\prlasto a_{1}],\dotsc)
\,,
\end{equation}
where $\MOR{[-]}{\IZ_{(p)}}{\IZ_{(p)}/p\IZ_{(p)}\cong\IF_p}$ denotes reduction modulo~$p$.

Writing each~$a_t$ in the basis~$\{V^i(1)\}$ and using the explicit formulas in \eqref{eq:F-for-Z(p)} and~\eqref{eq:R-for-Z(p)}, it is elementary to see that the equations~\eqref{eq:ker(F-R)} imply that
\[
p^{n-2}\Bigl(\ts\sum\limits_t\prlasto a_t\Bigr)=0
\,,
\qquad\text{and hence}\qquad
\ts\sum\limits_t[\prlasto a_t]=0
\,.
\]
Combining this and~\eqref{eq:connecting} we see that
\[
\im\partial\le\ker\nabla
\,,
\qquad
\text{where}
\quad
\MOR{\nabla}%
{\bigoplus_{\substack{j\in\IZ-p\IZ\\-(n-1)<t<\infty}}\IF_p}%
{\bigoplus_{          j\in\IZ-p\IZ                  }\IF_p}
\]
is the map that respects the $\bigoplus_{j\in\IZ-p\IZ}$ decomposition, and on each $j$-summand is defined by adding the $t$-components.

We now prove that also
\begin{equation}
\label{eq:ker=im}
\ker\nabla\le\im\partial
\,,
\end{equation}
which then immediately gives the short exact sequence in~\ref{i:ker(R)-Wh-TC}.

For any fixed~$T>-n+2$, consider the element
\[
z=(-1,0,\dotsc,0,1,0,\dotsc)\in
{\bigoplus_{                        -(n-1)<t<\infty }\IF_p}
\]
with $z_{-n+2}=-1$, $z_T=1$, and all other~$z_t=0$.
In order to prove~\eqref{eq:ker=im}, it is enough to show that each such~$z$ has a preimage under~$\partial$.

Assume that $T\ge0$.
Define~$a$ as in~\eqref{eq:a-in-Wh-TR} by the following formulas:
\begin{alignat*}{3}
&&
a_{-n+2}&=1V^{0}(1)
&&\in W_{1}(\IZ_{(p)})
\,,
\\
&&
a_{-n+3}&=pV^{0}(1)=Spa_{-n+2}
&&\in W_{2}(\IZ_{(p)})
\,,
\\
&&&\vdotswithin{=}&&\vdotswithin{\in W_{n-1}}
\\
&&
a_{-1}&=p^{n-3}V^{0}(1)=Spa_{-2}
&&\in W_{n-2}(\IZ_{(p)})
\,,
\\
&0\le t\le T-1
\qquad&
a_{t}&=p^{n-2}V^{0}(1)=Spa_{-1}
&&\in W_{n-1}(\IZ_{(p)})
\,,
\\
&&
a_{T}&=p^{n-2}V^{0}(1)-1V^{n-2}(1)
&&\in W_{n-1}(\IZ_{(p)})
\,,
\\
&&
a_{T+1}&=p^{n-2}V^{0}(1)-pV^{n-3}(1)=SFa_{T}
&&\in W_{n-1}(\IZ_{(p)})
\,,
\\
&&&\vdotswithin{=}&&\vdotswithin{\in W_{n-1}}
\\
&&
a_{T+n-3}&=p^{n-2}V^{0}(1)-p^{n-3}V^{1}(1)=SFa_{T+n-4}
&&\in W_{n-1}(\IZ_{(p)})
\,,
\\
&T+n-2\le t
\qquad&
a_{t}&=0=SFa_{T+n-3}
&&\in W_{n-1}(\IZ_{(p)})
\,.
\end{alignat*}
Using~\eqref{eq:ker(F-R)} we see that $a\in\ker(F-R)$, and \eqref{eq:connecting} then shows that~$\partial a=z$.
The case~$T<0$ is handled analogously.
This completes the proof of \autoref{keylemma}.
\end{proof}

\begin{example}
If $\pi_0\spec{A}\cong\IQ$ then $\tors_p\IQ=0=\IQ/p\,\IQ$ and so both dotted maps labeled~$R$ in~\eqref{eq:snake} are isomorphisms.
If $\pi_0\spec{A}\cong\IF_p$ then $\tors_p\IF_p=\IF_p=\IF_p/p\,\IF_p$ and the map~$\lambda$ in~\eqref{eq:snake} is the zero map.
Moreover, from the explicit formulas above one easily sees that
\[
\ker\bigl(
\MOR{F-R}{\wh_0^{\tr^n\!}(C)}{\wh_0^{\tr^{n-1}\!}(C)}
\bigr)
=
\ker\bigl(
\MOR{R}  {\wh_0^{\tr^n\!}(C)}{\wh_0^{\tr^{n-1}\!}(C)}
\bigr)
\]
when~$\pi_0\spec{A}\cong\IF_p$.
This implies that $R$ restricts to the zero map on~$\ker(F-R)$, the connecting map~$\partial$ in~\eqref{eq:snake} is an isomorphism, and so also
\[
\MOR{R}{\wh_{-1}^{\tc^n\!}(C)}{\wh_{-1}^{\tc^{n-1}\!}(C)}
\]
is an isomorphism.
These observations show that \autoref{lim-sum} below does not apply when $\pi_0\spec{A}$ is isomorphic to $\IQ$ or~$\IF_p$, and therefore in these cases our arguments do not produce counterexamples to surjectivity.
\end{example}

We are now ready to finish the proof of \autoref{not-surj-TR-TC}.

\begin{proof}[Proof of \autoref{not-surj-TR-TC}]
Assume that $G$ is either finitely generated free abelian or torsion-free hyperbolic, but not cyclic (and so $\VCyc=\Cyc$).
Consider the following commutative diagram.
\begin{equation}
\label{eq:asbl-Wh}
\begin{tikzcd}[column sep=-1em]
\ds EG(\Cyc)_+\sma_{\Or G}\Wh^{\TC}
\arrow[rr, "\asbl"]
\arrow[d, "\ts\one"']
&&
\ds\Wh^{\TC}(G)
\arrow[d, "\ts\two"']
\\
\ds EG(\Cyc)_+\sma_{\Or G}\holim_{n\in\IN}\Wh^{\TC^n}
\arrow[rr, "\asbl"]
\arrow[dr, "t"']
&&
\ds\holim_{n\in\IN}\Wh^{\TC^n\!}(G)
\\
&
\ds\holim_{n\in\IN}\Bigl(EG(\Cyc)_+\sma_{\Or G}\Wh^{\TC^n}\Bigr)
\arrow[ur, "\holim(\asbl)"' pos=.7, "\ts\three\!\!\!"]
\end{tikzcd}
\end{equation}
\autoref{Wh-holim} implies that the map~$\one$ is a $\pi_*$-isomorphism, and that the same is true for $\two$ under the assumptions on~$G$.
By \autoref{Wh-trick} and \autoref{pro-iso-TC}, the map~$\three$ is a $\pi_*$-isomorphism, too.

Recall the following two well-known facts.
First, given any sequence of spectra \(\dotsb\TO\T^{n+1}\TO\T^{n}\TO\dotsb\TO\T^2\TO\T^1\), for each~$q\in\IZ$ there is a natural short exact sequence
\[
0
\TO
\limone_{n\in\IN}\pi_{q+1}\T^n
\TO
\pi_q\holim_{n\in\IN}\T^n
\TO
\lim_{n\in\IN}\pi_q\T^n
\TO
0
\,;
\]
e.g., see~\cite{Bousfield-Kan}*{Theorem~IX.3.1, page~254}.
Second, if $\MOR{\T}{\Groupoids}{\IN\Sp}$ is a functor such that $\pi_q\T(C)=0$ for each $q<\ell$ and each $C\in\CF$, then there is a natural isomorphism
\[
\pi_\ell\Bigl(EG(\CF)_+\sma_{\Or G}\T\Bigr)
\cong
\colim_{C\in\obj\Or G(\CF)}\pi_\ell\T(C)
\,;
\]
e.g., see~\cite{RV-survey}*{Proof of Proposition~18, $\two$ to~$\five$, page~11}.
These assumptions are satisfied
for $\T=\Wh^{\TR^n},\Wh^\TR$ with~$\ell=0$,
and
for $\T=\Wh^{\TC^n},\Wh^\TC$ with~$\ell=-1$,
with respect to any family~$\CF$.

Now choose a complete set of representatives~$\CM$ of the conjugacy classes of maximal cyclic subgroups of~$G$.
The assumptions on~$G$ imply that any nontrivial cyclic subgroup is subconjugate to a unique $C\in\CM$, that for each~$C\in\CM$ the quotient $N_GC/Z_GC$ is trivial, and that $\CM$ is infinite.
This is obvious if $G$ is free abelian of rank at least~$2$; for hyperbolic groups, see \autoref{hyperbolic}.
Moreover, if~$C=1$ then $\wh_q^\tc(1)=0$.
So we get an isomorphism
\[
\colim_{C\in\obj\Or G(\Cyc)}\wh_q^\tc(C)
\cong
\bigoplus_{C\in\CM}         \wh_q^\tc(C)
\,.
\]

Putting these facts together with diagram~\eqref{eq:asbl-Wh}, we obtain the following commutative diagram with exact columns.
\begin{equation}
\label{eq:lim-sum}
\begin{tikzcd}[row sep=scriptsize, column sep=huge]
0
\arrow[d]
&
0
\arrow[d]
\\
\ds\smash{\bigoplus_{C\in\CM}}\limone_{n\in\IN}\!\!\wh_0^{\tc^n\!}(C)
\arrow[d]
\arrow[r]
&
\ds\limone_{n\in\IN}\!\wh_0^{\tc^n\!}(G)
\arrow[d]
\\
\ds\smash{\bigoplus_{C\in\CM}}\wh_{-1}^{\tc}(C)
\arrow[d]
\arrow[r, "\pi_{-1}\bigl(\asbl^{\Wh^{\TC}}\bigr)", "\ts\four"']
&
\wh_{-1}^{\tc}(G)
\arrow[d]
\\
\ds\smash{\bigoplus_{C\in\CM}}\limone[\phantom{1}]_{n\in\IN}\!\!\wh_{-1}^{\tc^n\!}(C)
\arrow[d]
\arrow[r, "\ts\five"']
&
\ds\lim_{n\in\IN}\smash{\bigoplus_{C\in\CM}}\wh_{-1}^{\tc^n\!}(C)\
\arrow[d]
\\
0
&
0
\end{tikzcd}
\end{equation}
By \autoref{Wh-trick}, we need to show that $\four$ is not surjective.
It is clearly enough to show that $\five$ is not surjective.

\autoref{keylemma}\ref{i:ker(R)-Wh-TC} states that the maps $\MOR{R}{\wh_{-1}^{\tc^{n\!}}(C)}{\wh_{-1}^{\tc^{n-1\!}}(C)}$ are surjective but not injective.
This implies that $\five$ is not surjective by \autoref{lim-sum} below, whose verification is an easy exercise, thus finishing the proof for~$\tc$.

For $\tr$ we proceed in the same way and obtain a diagram analogous to~\eqref{eq:lim-sum} with $\wh_0^{\tr^n}$ instead of~$\wh_{-1}^{\tc^n}$ and $\wh_1^{\tr^n}$ instead of~$\wh_0^{\tc^n}$.
We then invoke \autoref{keylemma}\ref{i:ker(R)-Wh-TR} (which is much easier to prove than~\ref{keylemma}\ref{i:ker(R)-Wh-TC}) to finish the proof.
\end{proof}

\begin{lemma}
\label{lim-sum}
Let
\begin{equation}
\label{eq:seq-ab-groups}
\dotsb
\TO
A_3
\TO
A_2
\TO
A_1
\TO
A_0
\end{equation}
be a sequence of abelian groups, and consider the direct sum of countably infinitely many copies of~\eqref{eq:seq-ab-groups}.
If each map $A_n\TO A_{n-1}$ is surjective but not injective, then the natural map
\[
\bigoplus^\infty\lim_{n\in\IN} A_n
\TO
\lim_{n\in\IN}\bigoplus^\infty A_n
\]
is not surjective.
\hfill\qedsymbol
\end{lemma}

\begin{remark}
\label{more-counterexamples}
The proof given above that $\pi_0(\asbl^\TR)$ is not surjective works for any \cplus{} symmetric ring spectrum~$\spec{A}$ whose homotopy groups are $\IZ_{(p)}$-modules and $\pi_0\spec{A}\neq0$ (under the stated assumptions on the group~$G$, of course).
The proof that $\pi_{-1}(\asbl^\TC)$ is not surjective works if moreover $\pi_0\spec{A}\cong\IZ_{(p)}$, or more generally if $\pi_0\spec{A}$ is a nontrivial free $\IZ_{(p)}$-module.

Moreover, we use the assumptions on~$G$ only to ensure that:
\begin{enumerate}[label=(\alph*)]
\item
there exists a compact~$BG$ (and hence $G$ is torsionfree and $\VCyc=\Cyc$);
\item
condition $\MNM{}{\{1\}}{\Cyc}$ holds (see \autoref{NM});
\item
for each maximal cyclic subgroup~$C\le G$ we have $N_GC/Z_GC=1$;
\item
there are infinitely many conjugacy classes of maximal cyclic subgroups.
\end{enumerate}
These four conditions on~$G$ are the only ones needed in the proof.
\end{remark}


\section{Example of an explicit computation}

In this section we prove \autoref{TC(A[S3];p)}, as well as a variant of it (\autoref{TC(A[S3];odd)}) that holds for odd primes~$p$.
\autoref{TC(A[S3];p)} is a direct consequence of \autoref{iso-TC-finite-groups} and \autoref{ES3(CYC)-smash-T} below.
We begin by fixing some notation.

We choose a cyclic subgroup of~$\Sym_3$ of order~$2$ and denote it~$C_2$, and we write~$C_3$ for~$A_3$.
We let $\MOR{i}{C_2}{\Sym_3}$ and $\MOR{p}{\Sym_3}{\Sym_3/C_3}$ be the inclusion and the projection homomorphisms.
The composition
\[
C_2\TO[i]\Sym_3\TO[p]\Sym_3/C_3
\]
is an isomorphism, and we define $\MOR{q=(pi)^{-1}p}{\Sym_3}{C_2}$.
There is a natural transformation from induction along~$i$ to restriction along~$q$:
\[
\begin{tikzcd}
\Top^{C_2}
\arrow[d, bend right, "i_*"' pos=.58]
\arrow[d, phantom,    "\stackrel{\tau}{\Rightarrow}"]
\arrow[d, bend left,  "q^*"  pos=.58]
\\
\Top^{\Sym_3}
\end{tikzcd}
\hspace{10em}
\begin{tikzcd}[column sep=scriptsize, row sep=-1ex]
\mathllap{\ds i_*X=\Sym_3\timesd_{C_2}X}
\arrow[r, "\tau_X"]
&
q^*X
\mathrlap{\,,}
\\
\mathllap{[g,x]}
\arrow[r, mapsto]
&
q(g)x
\mathrlap{\,.}
\end{tikzcd}
\]
Let \(\MOR{\pr}{EC_2}{\pt}\) be the projection.

\begin{lemma}
\label{ES3(CYC)}
There exists a $\Sym_3$-equivariant homotopy cocartesian square
\[
\begin{tikzcd}
i_*EC_2
\arrow[r, "\tau_{EC_2}"]
\arrow[d, "i_*\pr"']
&
q^*EC_2
\arrow[d]
\\
i_*\pt
\arrow[r]
&
E\Sym_3(\Cyc)
\mathrlap{\,.}
\end{tikzcd}
\]
\end{lemma}

\begin{proof}
Let $X$ be the $\Sym_3$-homotopy pushout of $i_*\pr$ and~$\tau_{EC_2}$.
It suffices to show that $X$, $X^{C_2}$, and~$X^{C_3}$ are contractible, and that $X^{\Sym_3}$ is empty.
Since $i_*\pr$ is a non-equivariant homotopy equivalence and $q^*EC_2$ is contractible, $X$ is also contractible.
The conditions on the fixed points are verified as follows.
\[
\biggl(
i_*\pt \FROM i_*EC_2 \TO q^*EC_2
\biggr)^H
=
\left\{
\setlength{\arraycolsep}{.15em}
\renewcommand{\arraystretch}{1.5}
\begin{array}{rcccll}
\pt&\FROM&\emptyset&\TO&\emptyset&\text{if $H=C_2$\,;}
\\
\emptyset&\FROM&\emptyset&\TO&EC_2\quad&\text{if $H=C_3$\,;}
\\
\emptyset&\FROM&\emptyset&\TO&\emptyset&\text{if $H=\Sym_3$\,.}
\end{array}
\right.
\qedhere
\]
\end{proof}

\begin{proposition}
\label{ES3(CYC)-smash-T}
Let $\MOR{\T}{\Groupoids}{\IN\Sp}$ be a functor that preserves equivalences.
Then there is a $\pi_*$-isomorphism
\[
\T(C_2)
\vee
\widetilde\T(C_3)_{hC_2}
\TO[\simeq]
E\Sym_3(\Cyc)_+\sma_{\Or\Sym_3}\T(\oid{\Sym_3}{-})
\,,
\]
where $C_2$ acts on~$C_3$ by sending the generator to its inverse, and $\widetilde\T(G)$ is the homotopy cofiber of the map
\(
\T(1)
\TO
\T(G)
\)
induced by the inclusion.
\end{proposition}

\begin{proof}
\autoref{ES3(CYC)} yields the following homotopy cartesian diagram.
\begin{equation}
\label{eq:ES3(CYC)-smash-T}
\begin{tikzcd}[column sep=large]
\ds{i_*EC_2}_+\sma_{\Or\Sym_3}\T(\oid{\Sym_3}{-})
\arrow[r, "\tau_{EC_2}\ssma\id"]
\arrow[d, "i_*\pr\ssma\id"']
&
\ds{q^*EC_2}_+\sma_{\Or\Sym_3}\T(\oid{\Sym_3}{-})
\arrow[d]
\\
\ds i_*\pt_+\sma_{\Or\Sym_3}\T(\oid{\Sym_3}{-})
\arrow[r]
&
\ds E\Sym_3(\Cyc)_+\sma_{\Or\Sym_3}\T(\oid{\Sym_3}{-})
\end{tikzcd}
\end{equation}
The bottom left-hand corner is isomorphic to~$\T(C_2)$.
We now identify the top horizontal map.

For any $C_2$-space~$X$ we have the following diagram.
\begin{equation}
\label{eq:tau_X-smash-id}
\begin{tikzcd}[column sep=huge]
\ds{X}_+\sma_{\Or C_2}\T(\oid{\Sym_3}{i_\diamond-})
\arrow[r, "\id\sma\T(\oid{\Sym_3}{\varsigma})"]
\arrow[d, "\cong"']
&
\ds{X}_+\sma_{\Or C_2}\T(\oid{\Sym_3}{q^\diamond-})
\arrow[d, "\cong"]
\\
\ds{i_*X}_+\sma_{\Or\Sym_3}\T(\oid{\Sym_3}{-})
\arrow[r, "\tau_{X}\ssma\id"']
&
\ds{q^*X}_+\sma_{\Or\Sym_3}\T(\oid{\Sym_3}{-})
\end{tikzcd}
\end{equation}
The vertical isomorphisms come from applying~\cite{kc}*{Fact~13.5(i), pages~994--995} to
\[
\begin{tikzcd}[column sep=small]
C_2/H
\arrow[d, "i_\diamond"', mapsto]
&
\Or C_2
\arrow[rr, hook]
\arrow[d, "i_\diamond"']
&&
\Top^{C_2}
\arrow[d, "i_*"]
&&
C_2/H
\arrow[d, "q^\diamond"', mapsto]
&
\Or C_2
\arrow[rr, hook]
\arrow[d, "q^\diamond"']
&&
\Top^{C_2}
\arrow[d, "q^*"]
\\
\Sym_3/H
&
\Or\Sym_3
\arrow[rr, hook]
\arrow[d, "\T(\oid{\Sym_3}{-})"]
\arrow[Rightarrow, rru, shorten <=1em, shorten >=1em, "\nu"]
&&
\Top^{\Sym_3}
&\ \text{and}\ &
\Sym_3/q^{-1}(H)
&
\Or\Sym_3
\arrow[rr, hook]
\arrow[d, "\T(\oid{\Sym_3}{-})"]
\arrow[Rightarrow, rru, shorten <=1em, shorten >=1em, "\upsilon"]
&&
\Top^{\Sym_3}
\\
&
\IN\Sp
&&&&&
\IN\Sp
\end{tikzcd}
\]
where $\nu$ and~$\upsilon$ are the obvious natural isomorphisms.
The natural transformation ${\tau}\colon{i_*}\Rightarrow{q^*}$ induces a compatible natural transformation ${\varsigma}\colon{i_\diamond}\Rightarrow{q^\diamond}$, and therefore it follows that \eqref{eq:tau_X-smash-id} commutes.
Finally, if $\MOR{\BF}{\Or C_2}{\IN\Sp}$ is any functor, then there are natural isomorphisms
\begin{equation}
\label{eq:free-smash-F}
{EC_2}_+\sma_{\Or C_2}\BF
\cong
{EC_2}_+\sma_{\Or C_2(1)}\BF_{|\Or C_2(1)}
\cong
\BF(C_2/1)_{hC_2}
\,.
\end{equation}
Combining \eqref{eq:ES3(CYC)-smash-T}, \eqref{eq:tau_X-smash-id}, and~\eqref{eq:free-smash-F}, we get the following homotopy cartesian square.
\begin{equation}
\label{eq:ES3(CYC)-smash-T-simplified}
\qquad
\begin{tikzcd}[column sep=large]
\ds\T(\oid{\Sym_3}{i_\diamond(C_2/1)})_{hC_2}
\arrow[r, "\T(\oid{\Sym_3}{\varsigma})_{hC_2}"]
\arrow[d]
&
\ds\T(\oid{\Sym_3}{q^\diamond(C_2/1)})_{hC_2}
\arrow[d]
\\
\ds\mathllap{\T(C_2)\simeq{}}i_*\pt_+\sma_{\Or\Sym_3}\T(\oid{\Sym_3}{-})
\arrow[r]
&
\ds E\Sym_3(\Cyc)_+\sma_{\Or\Sym_3}\T(\oid{\Sym_3}{-})
\end{tikzcd}
\end{equation}
The diagram of groupoids
\[
\hspace{6.33em}\begin{tikzcd}
\mathllap{\oid{\Sym_3}{i_\diamond(C_2/1)}={}}
\oid{\Sym_3}{\,(\Sym_3/1)}
\arrow[r, "\oid{\Sym_3}{\varsigma}"]
\arrow[d]
&
\oid{\Sym_3}{\,(\Sym_3/C_3)}
\mathrlap{{}=\oid{\Sym_3}{q^\diamond(C_2/1)}}
\arrow[d, "\sigma"]
&&&
{xC_3}\TO[g]{gxC_3}
\arrow[d, "\sigma", mapsto, shift right]
\\
1
\arrow[r, hook, "\incl"]
&
C_3
&&&
\!s(gxC_3)^{-1} gs(xC_3)
\end{tikzcd}
\hspace{-.25em}
\]
commutes, where $\MOR{s=i(pi)^{-1}}{\Sym_3/C_3}{\Sym_3}$.
The group~$C_2$ acts on~$C_2/1$ by right multiplication, and on~$C_3$ by conjugation.
With respect to these actions, the diagram of groupoids above is $C_2$-equivariant, and the vertical maps are non-equivariant equivalences.
Since $\T$ preserves equivalences, we obtain from~\eqref{eq:ES3(CYC)-smash-T-simplified} a homotopy cartesian square
\[
\begin{tikzcd}
\ds\T(1)_{hC_2}
\arrow[r, "\T(\incl)_{hC_2}"]
\arrow[d]
&
\ds\T(C_3)_{hC_2}
\arrow[d]
\\
\T(C_2)
\arrow[r]
&
\ds E\Sym_3(\Cyc)_+\sma_{\Or\Sym_3}\T(\oid{\Sym_3}{-})
\end{tikzcd}
\]
whose top horizontal map is evidently split by the projection~$C_3\TO1$.
Since homotopy orbits commute with homotopy cofibers, the proof of \autoref{ES3(CYC)-smash-T} is then completed by applying \autoref{cartesian-split} below.
\end{proof}

\begin{lemma}
\label{cartesian-split}
Assume that the commutative square of spectra
\[
\begin{tikzcd}
\BW
\arrow[r, "\alpha"]
\arrow[d]
&
\BX
\arrow[d]
\\
\BY
\arrow[r]
&
\BZ
\end{tikzcd}
\]
is homotopy cartesian and that $\alpha$ is split injective, i.e., there exists a $\MOR{\beta}{\BX}{\BW'}$ such that $\beta\alpha$ is a $\pi_*$-isomorphism.
Then there is a $\pi_*$-isomorphism
\[
\BY\vee\hocofib(\alpha)
\TO[\simeq]
\BZ
\,.
\]
\end{lemma}

\begin{proof}
Consider the following commutative diagram.
\[
\begin{tikzcd}[sep=small]
&
\BW
\arrow[rr, "\beta\alpha"]
\arrow[dd]
&
&
\BW'
\arrow[dd]
\\
\BW
\arrow[rr, crossing over, "\alpha" pos=.65]
\arrow[dd]
\arrow[ur, "\id"]
&
&
\BX
\arrow[ur, "\beta\!\!\!" pos=.25]
\\
&
\pt
\arrow[rr]
&
&
\pt
\\
\BY
\arrow[rr]
\arrow[ur]
&
&
\BZ
\arrow[from=uu, crossing over]
\arrow[ur]
\end{tikzcd}
\]
The front and the back squares are homotopy cartesian.
By taking homotopy fibers of the diagonal maps, and using the $\pi_*$-isomorphism $\hofib(\beta)\simeq\hocofib(\alpha)$, the result follows.
\end{proof}


We conclude with a variant of \autoref{TC(A[S3];p)} that holds only for odd primes~$p$.
First notice that, by the natural induction $\pi_*$-isomorphisms
\[
   X_+\sma_{\Or C_2}\T(\oid{C_2}{-})
\TO[\simeq]
i_*X_+\sma_{\Or\Sym_3}\T(\oid{\Sym_3}{-})
\]
(compare~\cite{kc}*{Theorem~12.2, page~988}),
the left vertical map in~\eqref{eq:ES3(CYC)-smash-T} is $\pi_*$-isomorphic to the classical assembly map for~$\T$
\[
{BC_2}_+\sma\T(1)
\cong
{EC_2}_+\sma_{\Or C_2}\T(\oid{C_2}{-})
\xrightarrow{\pr\ssma\id}
   \pt_+\sma_{\Or C_2}\T(\oid{C_2}{-})
\cong
\T(C_2)
\,.
\]
If $p\neq2$ then the trivial family for~$C_2$ is $p$-radicable, and therefore the classical assembly map for~$\TC(\spec{A}[C_2];p)$ is split injective by \autoref{inj-TC-technical}\ref{i:split-inj-TC-technical}.
Therefore, in this case, we may equally well apply \autoref{cartesian-split} to the left-hand vertical map in~\eqref{eq:ES3(CYC)-smash-T}, instead of to the top horizontal map as in the proof of \autoref{ES3(CYC)-smash-T}.
This proves the following result.

\begin{proposition}
\label{TC(A[S3];odd)}
If $p$ is an odd prime, then there is a $\pi_*$-isomorphism
\[
\Wh^\TC(C_2)
\vee
\TC(\spec{A}[C_3];p)_{hC_2}
\TO[\simeq]
\TC(\spec{A}[\Sym_3];p)
\,,
\]
where $\Wh^\TC(C_2)$ is the homotopy cofiber of the classical assembly map
\[
{BC_2}_+\sma\TC(\spec{A};p)
\TO
\TC(\spec{A}[C_2];p)
\,.
\]
\end{proposition}



\vfill
\end{document}